\documentclass[12pt,a4paper,reqno]{article}
\usepackage{amsfonts}
\usepackage{amssymb}
\usepackage{amsthm}
\usepackage{amsmath}
\usepackage{newlfont}
\usepackage{graphicx}
\usepackage{color}

\def\r{\mathbb R}

\def\l{\mathbb L}
\def\h{\mathbb H}

\newtheorem{theorem}{Theorem}[section]
\newtheorem{definition}[theorem]{Definition}
\newtheorem{proposition}[theorem]{Proposition}
\newtheorem{remark}[theorem]{Remark}
\newtheorem{corollary}[theorem]{Corollary}
\newtheorem{example}[theorem]{Example} 

\title{The two-dimensional analogue of the  Lorentzian catenary and the Dirichlet problem}
\author{Rafael L\'opez\footnote{Partially supported by the grant no. MTM2017-89677-P, MINECO/AEI/FEDER, UE.}\\
Departamento de Geometr\'{\i}a y Topolog\'{\i}a\\
 Instituto de Matem\'aticas (IEMath-GR)\\
 Universidad de Granada\\
 18071 Granada, Spain\\
\texttt{rcamino@ugr.es}
}

\date{}

\begin{document}
\maketitle
\begin{abstract} 

We generalize in   Lorentz-Minkowski space $\l^3$ the two-dimensional analogue of the catenary of Euclidean space. We solve the Dirichlet problem when the bounded domain is mean convex and the boundary data has a spacelike extension to the domain. We also classify all singular maximal surfaces of $\l^3$ invariant by a uniparametric group of translations and rotations.
\end{abstract}


 \section{Introduction and motivation}

The purpose of this paper is to investigate the physical problem of characterizing the surfaces     in  Lorentz-Minkowski space with lowest gravity center and solve the corresponding Dirichlet problem.  The existence of a variety of causal vectors in the Lorentzian setting makes that appear several issues that need to be fixed. Firstly, we recall this problem in the Euclidean space  in order to motivate our definitions.  Let $\r^2$ be the  Euclidean plane with  canonical coordinates $(x,y)$ where  the $y$-axis indicates the gravity direction. Consider the physical problem of finding the curve in the halfplane $y>0$ with the lowest gravity center.  If  the curve is $y=u(x)$, then $u$ satisfies the equation
 \begin{equation}\label{cat}
\frac{u''}{1+u'^2}=\frac{1}{u}.
\end{equation}
The solution of this equation is  known the catenary
 $$u(x)=\frac{1}{a}\cosh(ax +b), \ a,b\in\r, a\not=0.$$
Equation \eqref{cat} can be expressed in terms of the curvature $\kappa$ of the curve as 
\begin{equation}\label{cat0}
\kappa=\frac{\langle \textbf{n},\vec{a}\rangle}{y},
\end{equation}
  where $\textbf{n}$   is the unit normal vector  and $\vec{a}=(0,1)$. In particular, equation \eqref{cat0} prescribes the angle that makes the vector $\textbf{n}$ with the vertical direction.
  
The   generalization in Euclidean $3$-space $\r^3$  of the property of the catenary is to find surfaces in the halfspace $z>0$ with the lowest gravity center. If  $(x,y,z)$ denote  the canonical coordinates of $\r^3$ and $z$ indicates the direction of the gravity, these surfaces characterize by means of   the equation
$$H=\frac{\langle N,\vec{a}\rangle}{z},$$
  where $H$ is the mean curvature of the surface and $\vec{a}=(0,0,1)$.   The surface is called in the literature the {\it two-dimensional analogue of the catenary} (\cite{bht,dh}).   Historically, this problem goes back to early works of Lagrange and Poisson on the equation that models a heavy surface in vertical gravitational field. If  we embed $\r^2$ as the $xz$-plane by identifying the $y$-axis of $\r^2$ with the $z$-axis of $\r^3$, and   we rotate the   catenary with respect to the $x$-axis, we obtain the catenoid $a^2(y^2+z^2)=\cosh^2(x)$, which is  the only non-planar rotational minimal surface of $\r^3$. 
  
  More general,  given a  constant $\alpha\in\r$,   a surface in the halfspace $z>0$ is called  a singular minimal surface if satisfies
\begin{equation}\label{cat1}
H=\alpha\frac{\langle N,\vec{a}\rangle}{z}.
\end{equation}

The theory of singular minimal surfaces has been intensively studied from the works of Bemelmans, Dierkes and Huisken, among others. Without to be a complete list, we refer to    \cite{bd,bht,di,di2,dh,lo3,lo4,lo6,ni2}. 

Once presented the problem in the Euclidean space, we  proceed to  generalize it in  the Lorentz-Minkowski space.  As in the Euclidean case, we begin with the one-dimensional case. Let $\l^2$ be the Lorentz-Minkowski plane defined as the affine $(x,y)$-plane $\r^2$ endowed with   the metric $dx^2-dy^2$. Here we use the  usual terminology of the Lorentz-Minkowski space: see \cite{on} as a general reference and   \cite{lo1} for curves and surfaces in Lorentz-Minkowski space. In what follows, we will assume that for a given set, the causal character is the same in all its points, that is, we do not admit the existence of   points with different causal character. 

A first issue     is that  in $\l^2$ it does not make sense the notion of gravity in $\l^2$ because the $y$-coordinate represents the time  in the Lorentzian context. Thus we need to view the initial  problem as a problem of finding curves in $\l^2$ with prescribed angle between the normal vector and a fixed direction, such as it was shown in  equation \eqref{cat0}. There appear  two new  issues.
Firstly there are three types of curves in $\l^2$ according its causal character, namely, spacelike, timelike and  lightlike and the behavior of each of these curves is completely different. Because  our interest is to  keep the Riemannian sense,  we will only consider spacelike curves. 

 A second issue  is the choice of the axis with respect to what we measure the angle of the normal vector $\textbf{n}$. Notice that in Euclidean plane both axes are indistinct  but in $\l^2$ the $y$-axis and the $x$-axis are not interchangeable by a rigid motion.  Thus it arises the problem what axis to be fixed. Since for a spacelike curve, the vector $\textbf{n}$  is   timelike, we will measure the angle between $\textbf{n}$ and  the $y$-axis, which is also timelike. This  is also justified because   it makes sense to define the angle between two timelike vectors (\cite[p.144]{on}).   After all these considerations, let us proceed.  

Let  $\gamma=\gamma(s)$  be  a spacelike curve parametrized by the arc-length $s\in I$ and contained in the halfplane $y>0$ of $\l^2$. The curvature $\kappa$ of $\gamma$ is defined by $\gamma''(s)=\kappa(s)\textbf{n}(s)$ where $\textbf{n}$ is a unit normal vector of $\gamma$. Here we are assuming $\kappa\not=0$.  Motivated by the  equation \eqref{cat0}, we ask for those spacelike curves of $\l^2$ that satisfy the same equation \eqref{cat0} where $\textbf{a}=(0,1)$.   
If $\gamma$ is a graph  $y=u(x)$,  then $\gamma(x)=(x,u(x))$, which is not parametrized by the arc-length. Then $\textbf{n}=(u',1)/\sqrt{1-u'^2}$, $\langle \textbf{n},\vec{a}\rangle=-1/\sqrt{1-u'^2}$  and 
$$\kappa(x)=-\frac{1}{1-u'^2}\langle\gamma''(x),\textbf{n}(x)\rangle=\frac{u''(x)}{(1-u'(x)^2)^{3/2}}.$$
Let us observe that $u'^2<1$ because $\gamma$ is a spacelike curve.
Equation  \eqref{cat0} is    now 
\begin{equation}\label{eq0}
\frac{u''}{1-u'^2}=-\frac{1}{u},
\end{equation}
which will be the Lorentzian model of \eqref{cat} that we are looking for. The spacelike condition $u'^2-1<0$ is an extra hypothesis comparing with the Euclidean case. For example,   $u(x)= \sinh(x)$, with $u>0$, solves \eqref{eq0}, but   $u'^2>1$. So,   the corresponding curve $y=u(x)$ is a timelike curve. In contrast, because we are assuming that the curve is spacelike,   the right solution of \eqref{eq0} is   
\begin{equation}\label{eq00}
u(x)=\frac{1}{a}\sin(ax+b),\quad x\in\left(-\frac{b}{a},\pi-\frac{b}{a}\right),
\end{equation}
where $a\not=0\ a,b\in\r$. This curve will be the analogue catenary in $\l^2$. As in the Euclidean case,  we introduce a constant $\alpha\in\r$ and we consider the analogous equation of \eqref{cat0}, namely, 
\begin{equation}\label{cat11}
\kappa=\alpha\frac{\langle\textbf{n},\vec{a}\rangle}{\langle p,\vec{a}\rangle}=-\alpha\frac{\langle\textbf{n},\vec{a}\rangle}{y},
\end{equation}
where $p=(x,y)\in\l^2$.  For instance,  the   curve \eqref{eq00} is the solution   for $\alpha=-1$. 

 Following the same steps done in the Euclidean setting, we embed  $\l^2$ in the Lorentz-Minkowski $3$-space $\l^3$.  Here $\l^3$ is the affine  Euclidean $3$-space endowed with the metric $dx^2+dy^2-dz^2$. Then  $\l^2$ is identified with the $xz$-plane, the $y$-axis of $\l^2$  with the $z$-axis of $\l^3$  and the vector $(0,1)\in\l^2$ with $\vec{a}=(0,0,1)$.
Definitively, the objects of our study in this paper are described in the following definition.

\begin{definition} Let $\alpha$ be a nonzero real number. A spacelike surface $S$ in the halfspace $z>0$ of  $\l^3$ is called an  $\alpha$-singular maximal surface if satisfies 
\begin{equation}\label{eqL}
H(p)=\alpha\frac{\langle N(p),\vec{a}\rangle}{\langle p,\vec{a}\rangle}=-\alpha\frac{\langle N(p),\vec{a}\rangle}{z}\quad (p\in S),
\end{equation}
where $N$ is a unit normal vector field on $S$ and   $H$ is the mean curvature.
\end{definition}

Here $H$ the trace of the second fundamental form of $S$, that is, the sum of the principal curvatures. We will omit the constant $\alpha$ if it is understood in the context. Recently, these surfaces have been studied in \cite{mt} relating  the Riemannian and the Lorentzian settings by means of   a Calabi type correspondence.

   In view of \eqref{cat0}, and as a motivation of this paper, the case $\alpha=-1$ in equation \eqref{eqL} is the corresponding   {\it two-dimensional analogue  of the Lorentzian catenary}.  Other known examples appear when $\alpha=2$ because in such a case,  the surface is a minimal surface in the steady state space (\cite{lo2}). Another special example is  the hyperbolic plane $\h^2(r)=\{p\in\l^3:\langle p,p\rangle=-r^2, z>0\}$, $r>0$. This surface has mean curvature $H=2/r$ for $N(p)=p/r$. It is clear that $\h^2(r)$ satisfies \eqref{eqL} for $\alpha=2$. Even more, $\h^2(r)$ satisfies \eqref{eqL} {\it for any} vector $\vec{a}$.

On the other hand, we extend a similar property that has the catenary   in Euclidean space. Indeed, we take 
  the catenary \eqref{eq00} and we rotate   with respect to the $x$-axis. The rotations that leave   pointwise fixed the $x$-axis are described by 
$$\left\{\left(\begin{array}{ccc}1&0&0\\ 0 &\cosh\theta&\sinh\theta\\ 0&\sinh\theta&\cosh\theta\end{array}\right):\theta\in\r\right\}.$$
 For a curve $z=u(x)$, namely, $\gamma(x)=(x,0,u(x))$, $x\in I\subset\r$,  contained in the $xz$-plane, the corresponding rotational surface $S$ is parametrized by 
\begin{equation}\label{eqx}
X(x,\theta)=(x,u(x)\sinh\theta ,u(x)\cosh\theta  ), \ \theta\in \r.
\end{equation}
If $u(x)=\sin(ax+b)/a$, it is not difficult to see that the corresponding rotational surface \eqref{eqx} has zero mean curvature, that is, $S$ is a maximal surface of $\l^3$.  This surface is called in the literature the catenoid of second kind  or the hyperbolic catenoid. 

\begin{remark}If we rotate the curve $u(x)=  \sinh(ax+b)/a$, the timelike solution of \eqref{eq0}, with respect to the $x$-axis,   the rotational surface is a timelike surface with zero mean curvature (\cite{lo0}). Similarly, any  vertical straight line  is a timelike curve that satisfies \eqref{cat0} and if  we rotate with respect to the $x$-axis, we obtain a (timelike) plane parallel to the $yz$-plane, which has zero mean curvature everywhere.
 \end{remark}
 
As a conclusion, the generalization in $\l^3$ of the two-dimensional analogue of the catenary, or more generally,   singular maximal surfaces in Lorentz-Minkowski space $\l^3$, is carried out for spacelike surfaces and the angle between $N$ and $\vec{a}$ is measured with respect to the (timelike) $z$-axis. We have also discussed that there are other possibilities to generalize the initial problem in  $\l^3$, although all them less justified, as for example, changing the axis     $\vec{a}=(0,0,1)$  by   $(1,0,0)$ (spacelike) or $(1,0,1)$ (lightlike). Also, we may consider timelike surfaces and measuring the angle between $N$ with respect to an axis of $\l^3$. 

In this paper  we will also be interested to solve the Dirichlet problem of the singular maximal surface equation. Since  a spacelike surface is locally  the graph of a function $z=u(x,y)$, the nonparametric form of equation \eqref{eqL} is 
 \begin{equation}\label{eq2}
 \mbox{div}\frac{D  u}{\sqrt{1-|D  u|^2}}=\alpha\frac{1}{u\sqrt{1-|D  u|^2}},
 \end{equation}
together the spacelike condition $|Du|<1$.  The left-hand side of this equation is the mean curvature of the graph $z=u(x,y)$ computed with respect to the 
upwards orientation
 $$N=\frac{1}{\sqrt{1-|Du|^2}}(Du,1).$$
  Comparing \eqref{eq2} with the Riemannian case (\cite{di,di1,di2,lo5}), this equation is not uniformly elliptic and, as a consequence, this requires   to ensure that $|Du|$ is bounded away from $1$.
 
 This paper is organized as follows. In Section \ref{sec2} we classify all  singular maximal surfaces that are invariant by a uniparametric group of translations and  of rotations. In Section \ref{sec3} we describe the solutions of \eqref{eqL} that are invariant by rotations about the $z$-axis and finally, in Section \ref{sec4} we solve the Dirichlet problem associated to equation \eqref{eq2} for  mean convex domains and   arbitrary boundary data.

\section{Invariant singular maximal surfaces}\label{sec2}

In this section we classify and describe  all singular maximal surfaces that are invariant by a uniparametric group of translations or of rotations   of $\l^3$. Firstly, we notice that  some transformations of the affine Euclidean space $\r^3$    preserve the singular maximal surface equation. To fix the terminology, a vector $\vec{v}\in\r^3$ is called horizontal direction if it is parallel to the $xy$-plane and it is called   vertical  if is parallel to the $z$-axis. 

It is clear that a solution of \eqref{eqL} is   invariant by a translation along a horizontal direction, that is, if $S$ is an $\alpha$-singular maximal surface, then $S+\vec{v}$ is also an $\alpha$-singular maximal surface, where $\vec{v}$ is a horizontal vector of $\r^3$. Similarly, the same property holds if we rotate $S$ with respect to a vertical direction because the term $\langle N,\vec{a}\rangle $ and the denominator $z$ in \eqref{eqL} are invariant by this type of rotations. Finally,  if $\lambda>0$ is a positive real number, and $T_\lambda(p)=p_0+\lambda(p-p_0)$ is the dilation with center $p_0\in\r^2\times\{0\}$, then $T_\lambda(S)$ is an $\alpha$-singular maximal surface.  

\begin{remark}\label{re1}
We point out  that a rigid motion of $\l^3$ does not preserve in general the equation \eqref{eqL} because the denominator $z$ may change in general by the motion.
\end{remark}

As we have announced, a  natural source of examples of singular maximal surfaces of $\l^3$ finds in the class of invariant surfaces by a uniparametric group of rigid motions. The key point is that equation \eqref{eqL}, which locally is the partial differential equation \eqref{eq2}, changes into an ordinary differential equation. In particular, by standard theory, there always is a solution  for any initial conditions.

\subsection{Surfaces invariant by translations}

We begin the study of the   surfaces invariant by a uniparametric group of translations. Since the rulings generated by this group are straight lines contained in the surface, and the surface is spacelike, then any ruling is a spacelike line. Thus the vector generating the group of translation must be spacelike.  Let   $\vec{v}$ be a unit spacelike vector and consider a surface $S$  invariant by the group of translations generated by $\vec{v}$. Then $S$ parametrizes as 
$X(s,t)=\gamma(s)+t\vec{v}$, where $\gamma$ is a planar  spacelike curve of $\l^3$ contained in a (timelike)  orthogonal plane to $\vec{v}$. Equation \eqref{eqL} is 
$$ \kappa\,\mbox{det}(\gamma',\vec{v},\textbf{n})=\alpha\frac{\mbox{det}(\gamma',\vec{v},\vec{a})}{\gamma_3+tv_3},$$
where $\gamma=(\gamma_1,\gamma_2,\gamma_3)$ and $\vec{v}=(v_1,v_2,v_3)$.   We consider the orientation in $\gamma$  so $\gamma'\times\vec{v}=\textbf{n}$. Since $\textbf{n}$ is a unit timelike vector,  the above equation is now 
\begin{equation}\label{ejt}
\kappa  (\gamma_3+tv_3)+\alpha\langle\textbf{n},\vec{a}\rangle=0.
\end{equation}
This is a polynomial equation on $t$, hence 
$$\kappa v_3=0,\quad \kappa  \gamma_3 +\alpha\langle\textbf{n},\vec{a}\rangle=0.$$
Since $\kappa\not=0$, we deduce  that $v_3=0$ and  
 $\kappa\gamma_3+\alpha\langle\textbf{n},\vec{a}\rangle=0$. Then $\vec{v}$ is a horizontal vector and $\gamma$ is a planar curve contained in a vertical plane. After a horizontal translation and a rotation about the $z$-axis, we assume that this plane is the $xz$-plane  which can be identified with $\l^2$. Furthermore, the equation  $\gamma_3 +\alpha\langle\textbf{n},\vec{a}\rangle=0$  means that  $\gamma$ satisfies, as a planar curve of $\l^2$,  the one-dimensional singular maximal surface equation \eqref{cat11}. The converse of this result is immediate. 
 
 \begin{proposition} \label{pr211}
 Let $S$ be an $\alpha$-singular maximal surface of $\l^3$ invariant by a uniparametric group of translations generated by $\vec{v}$ and denote by $\gamma$ its generatrix. Then   $\vec{v}$ is a horizontal vector,  $\gamma$ is contained in a plane orthogonal to $\vec{v}$ and $\gamma$, as a planar curve, satisfies \eqref{cat11}. Conversely,    if $\gamma$ is a curve in $\l^2$ that satisfies \eqref{cat11} and, if  we embed this curve in the $xz$-plane as usually, then the surface $X(s,t)=\gamma(s)+t(0,1,0)$  is an $\alpha$-singular maximal surface.
 \end{proposition}

In view   of this proposition, consider the one-dimensional case of equation \eqref{eqL}.  Let   $\gamma(s)=(x(s),y(s))$ be a spacelike curve in $\l^2$ that satisfies \eqref{cat11}. Since $\gamma$ is spacelike, then $x'^2-y'^2>0$, in particular, $x'(s)\not=0$ for every $s$ and thus $\gamma$ is globally the graph of a function $u=u(x)$, $x\in I\subset\r$. Equation \eqref{cat11} is now
\begin{equation}\label{eq-one}
\frac{u''}{1-u'^2}=\alpha\frac{1}{u},\quad u>0,  u'^2<1.
\end{equation}
It is possible to find some explicit solutions of \eqref{eq-one} by simple quadratures.  In the Introduction we have seen that if  $\alpha=-1$, the solution is $u(x)= \sin(ax+b)/a$, where $a\not=0$, $a,b\in\r$ and  where $x$ is defined in some interval to ensure that $u>0$. If $\alpha=1$, it is easy to find that the solution of \eqref{eq-one} is  
 $$u(x)=\frac{1}{a}\sqrt{1+a^2 x^2+2abx+b^2},\ a,b\in\r, a>0.$$
After a change of variable,   this function $u$ writes as $u(x)=\sqrt{1+a^2x^2}/a$, $a>0$. It is immediate that    $u$ is the upper branch of the  hyperbola $a^2(x^2-y^2)=-1$. This curve,  viewed as a planar curve in $\mathbb{L}^2$, has  nonzero constant curvature $\kappa=a$. The generated surface by Proposition \ref{pr211} is the right-cylinder of equation $a^2(x^2-z^2)=-1$. 

\begin{remark}\label{rem22}
Such as it was done for the catenary $u(x)=\sin(ax+b)/a$,   if we rotate the   curve $u(x)=\sqrt{1+a^2x^2}/a$  with respect to the $x$-axis,  we obtain the hyperbolic plane   $\h^2(1/a)$. \end{remark}

\begin{remark} Similarly as in the case $\alpha=-1$, there is a timelike solution of \eqref{eq-one} by replacing the spacelike condition $u'^2<1$ by $u'^2>1$. The solution if now $u(x)=\sqrt{a^2x^2-1}/a$, where $a>0$ and $x>1/a$. The function $u$  is the positive part of the hyperbola $x^2-y^2=1/a^2$, which is a timelike curve. If we rotate about the $x$-axis, the generated surface is $x^2+y^2-z^2=1/a^2$. This surface is the (upper part of) de Sitter space $\mathbb{S}^2_1(1/a)=\{p\in\l^3:\langle p,p\rangle=1/a^2\}$. This surface satisfies \eqref{eqL} when $\alpha=2$ and plays the same role than the hyperbolic plane in the family of timelike surfaces of $\l^3$.
\end{remark}

We now describe the geometric properties of the solutions of \eqref{eq-one}. See figure \ref{fig1}.

\begin{theorem} Let $u=u(x)$ be a solution of \eqref{eq-one}, $x\in I$, where $I\subset\r$ is the maximal domain of $u$. Then $u$ is symmetric about a vertical line and $I=\r$ if $\alpha>0$ or $I$ is a bounded interval if $\alpha<0$. Furthermore:
\begin{enumerate}
\item Case $\alpha>0$. The function $u$ is convex with a unique global minimum, $\lim_{r\rightarrow \infty}u(r)=\infty$ and  $\lim_{r\rightarrow \infty}u'(r)=1$.
\item Case $\alpha<0$. The function $u$ is concave with a unique   global maximum. If $I=(-b,b)$, then $\lim_{r\rightarrow b}u(r)=0$ and $\lim_{r\rightarrow b}u'(r)=-1$.
\end{enumerate}
\end{theorem}

\begin{proof}
If $u$ has a critical point at $r=r_o$, then $u''(r_o)=\alpha/u(r_o)$ has the same sign than $\alpha$. Hence, there is one critical point at most that will be a global minimum (resp. maximum) if $\alpha>0$ (resp. $\alpha<0$).

{\it Claim: There exists a critical point of $u$.}

Suppose now that the claim is proved and we finish the proof of theorem. After a change in the variable $x$, we  suppose that $x=0$ is the critical point, $u'(0)=0$. Then $u$ is the solution of \eqref{eq-one} with initial conditions $u(0)=u_0>0$ and $u'(0)=0$. It is clear that $u(-s)$ is also a solution of the same initial value problem, so $u(s)=u(-s)$ by uniqueness. This proves that $u$ is symmetric about the $y$-axis.

Multiplying \eqref{eq-one} by $u'$, we obtain a first integral 
\begin{equation}\label{eq-f}
\frac{1}{1-u'^2}=\mu  u^{2\alpha},
\end{equation}
for some positive constant $\mu >0$.

\begin{enumerate}
\item Case $\alpha>0$.  Since $u(x)\geq u_0$,  we deduce from \eqref{eq-one} that  $u'$  and $u''$ are  bounded functions and this implies that the maximal domain is $\r$. Since $u$ is a convex function, then $u(r)\rightarrow\infty$ as $r\rightarrow\infty$ and from \eqref{eq-f}, we conclude that  $u'(r)\rightarrow 1$ as $r\rightarrow\infty$.
\item Case $\alpha<0$. By symmetry, $I=(-b,b)$ for some $b\leq\infty$. Since $u$ is a positive concave function, then $b<\infty$. Using the concavity of $u$ again, and because $u'^2<1$, then the graph of $u$ must meet the $x$-axis, that is, $\lim_{r\rightarrow b}u(r)=0$. From \eqref{eq-f}, we deduce $\lim_{r\rightarrow b}u'(r)^2=1$, and by concavity, $\lim_{r\rightarrow b}u'(r)=-1$.
\end{enumerate}

We now prove the claim. The proof is by contradiction. Assume that the sign of $u'$ is constant  and denote $I=(a,b)$  with $-\infty\leq a<b\leq\infty$.

\begin{enumerate}
\item Case $\alpha>0$. We   suppose that $u'>0$ in $I$ (similar argument if $u'$ is negative). Since $u$ is increasing and $u'$  and $u''$ are bounded close $r=b$, we deduce that $b=\infty$ by standard   theory. If $-\infty<a$, then $\lim_{r\rightarrow a}u(r)=0$ because  on the contrary, we could extend $u$ beyond $r=a$ because $u'$ and $u''$ would be  bounded functions. Therefore $\lim_{r\rightarrow a}u'(r)^2=1$ by \eqref{eq-f}. Since $u'>0$, this limit is just $1$.   This is a contradiction because $u'$ is an increasing function  and we would have $u'>1$ in $I$, which is not possible by the spacelike condition.

Thus $a=-\infty$. Since $u$ is increasing and $u>0$ in $\r$, we find $\lim_{r\rightarrow -\infty}u(r)=c\geq 0$. Because $u'>0$ and $u''>0$, then   $\lim_{r\rightarrow -\infty}u'(r)=\lim_{r\rightarrow -\infty}u''(r)=0$. However, by \eqref{eq-one},    and letting $ r\rightarrow -\infty$, we have $u''(r)$ goes to  $\alpha/c\not=0$ if $c>0$ or to $\infty$ if $c=0$, obtaining  a contradiction. 
\item Case $\alpha<0$. We   suppose that $u'>0$ in $I$ (similar argument if $u'$ is negative). Since $u'$  and $u''$ are bounded for $r$ close to $b$, then $b=\infty$ and by concavity,  we deduce that $-\infty<a$. If $u$ is bounded from above with $\lim_{r\rightarrow \infty}u(r)=c>0$, then $\lim_{r\rightarrow \infty}u'(r)=0$ and since $u''<0$, then $\lim_{r\rightarrow \infty}u''(r)=0$. By  \eqref{eq-one}, we find $\lim_{r\rightarrow \infty}u''(r)=\alpha/c<0$, a contradiction. Thus $\lim_{r\rightarrow \infty}u(r)=\infty$. By using \eqref{eq-f}, we conclude $\lim_{r\rightarrow \infty}u'(r)^2=1$, so this limit is $1$: a contradiction because   $u'$ is a decreasing function and we would have $u'>1$ in the interval $I$, which is not possible.
\end{enumerate}
\end{proof}

 \begin{figure}[hbtp]
\centering \includegraphics[width=.4\textwidth]{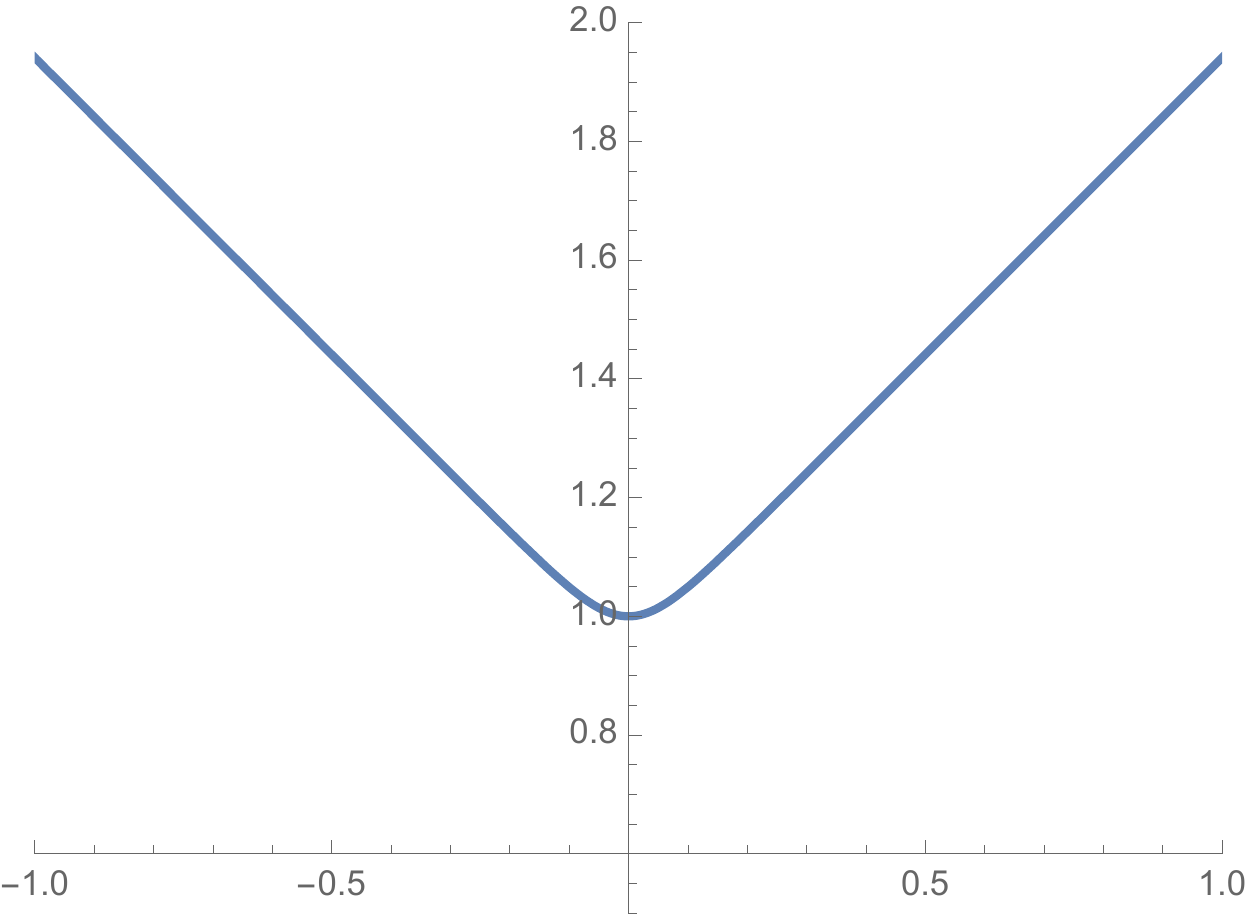}\quad \includegraphics[width=.4\textwidth]{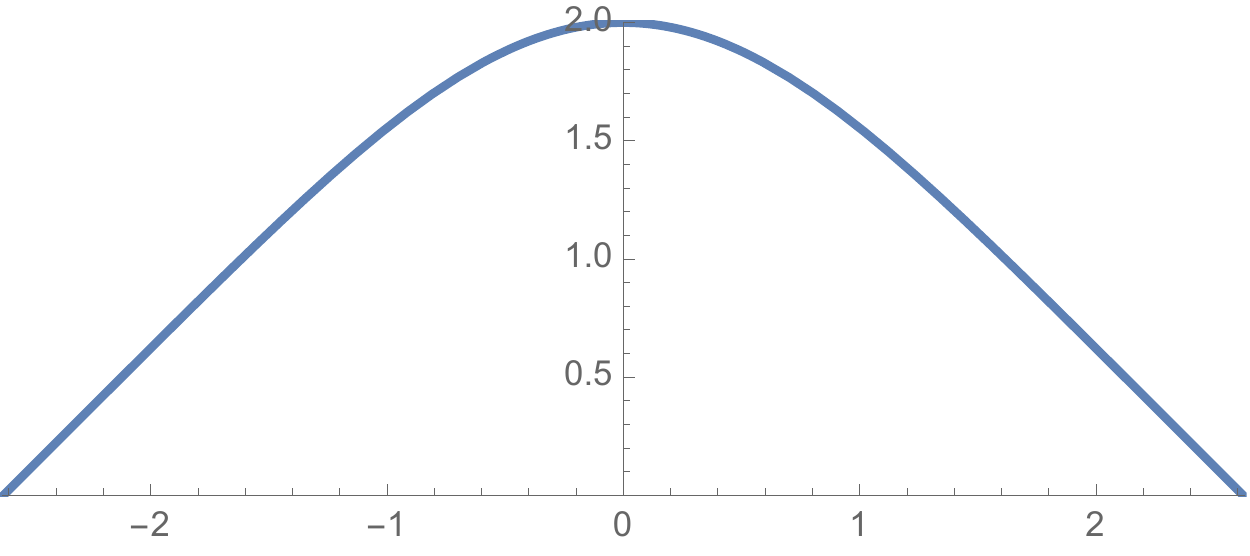}
 \caption{Solutions of \eqref{eq-one}. Left:  $\alpha=1$. Right:  $\alpha=-2$ }\label{fig1}
 \end{figure}
\subsection{Surfaces of revolution   with respect to a spacelike axis and a lightlike axis}

The second source of examples of  singular maximal surfaces are the surfaces   invariant by a uniparametric group of rotations. A  difference between the Euclidean and the Lorentzian settings is that in $\l^3$ there are three types of surfaces of revolution depending if the rotational axis is spacelike, timelike or lightlike. Section \ref{sec3} is devoted to the surfaces of revolution whose    rotation axis is timelike because this type of surfaces   will   play a special role  in   the solvability of the Dirichlet problem in Section \ref{sec4}. In  this section we investigate the  cases that  the rotation axis is spacelike and lightlike. 
 
We point out that  there is not an {\it a priori} relation between the rotation axis $L$ and the vector $\vec{a}=(0,0,1)$ of equation \eqref{eqL}. This implies that if we apply a rigid motion to prescribe the rotation axis, then the vector $\vec{a}$ does change: see also Remark \ref{re1}.

Firstly we consider the case that the axis is spacelike. 

\begin{proposition}\label{pr-x}
 Let $S$ be a spacelike surface   of $\l^3$  invariant by  the uniparametric group of rotations about a spacelike axis $L$. Suppose that $S$ satisfies equation \eqref{eqL} where   $\vec{a}$ is a timelike vector. Then either $\vec{a}$ is orthogonal to $L$, or $S$ is the hyperbolic plane $\h^2(r)$ being $\vec{a}$   an arbitrary timelike vector.
\end{proposition}

\begin{proof} After a rigid motion of $\l^3$ we assume that $L$ is the $x$-axis. This rigid motion changes the vector $\vec{a}$ in equation \eqref{eqL} and $\vec{a}$ must be considered an arbitrary (timelike) vector. Let $\vec{a}=(a,b,c)$   denote the new vector $\vec{a}$ in \eqref{eqL} after the rigid motion. Since $\vec{a}$ is timelike, then $c\not=0$. 

Using  the expression of a parametrization \eqref{eqx} of $S$ and after some computations, equation  \eqref{eqL} is a polynomial equation on  $\{1,\sinh\theta,\cos\theta\}$. Since these functions are   linearly independent, all three  coefficients  (which are functions on the  variable $s$) must vanish, obtaining 
$$-c(1-u'^2+uu'')+\alpha c(1-u'^2)=0$$
$$-b(1-u'^2+uu'')+\alpha b(1-u'^2)=0$$
$$as(1-u'^2+uu'')-\alpha auu'(1-u'^2)=0.$$
Since $c\not=0$,   we find  $a(uu'-s)=b(uu'-s)=0$. If $uu'-s\not=0$, then $a=b=0$, proving that $\vec{a}=(0,0,c)$, hence $L$ is orthogonal to the $x$-axis and the result is proved. The other possibility is $uu'-s=0$. Solving this equation, we find $u(s)=\sqrt{s^2+r^2}$, $r>0$. Then  $X(s,\theta)=(s,\sqrt{s^2+r^2}\sinh\theta,\sqrt{s^2+r^2}\cos\theta)$ and it is immediate that this surface is the hyperbolic plane $\h^2(r)$. 
\end{proof}

As a consequence of Proposition \ref{pr-x},   and besides the hyperbolic plane as a special case, we can assume  that $\vec{a}=(0,0,1)$ in the singular maximal surface equation \eqref{eqL},          and that the rotation axis is the $x$-axis.  In such a case, the proof of Proposition \ref{pr-x} gives immediately that equation \eqref{eqL} is
$$\frac{u''}{1-u'^2}=(\alpha-1)\frac{1}{u}.$$
This equation  is just the equation \eqref{eq-one}. Identifying the Lorentzian plane $\l^2$ with the plane of equation $y=0$, we have obtained the following result. 

\begin{proposition}\label{pr27} Any rotational $\alpha$-singular maximal surface in $\l^3$  about the $x$-axis is generated by a planar curve in $\l^2$ that  satisfies the one-dimensional $(\alpha-1)$-singular maximal surface equation. Conversely, any planar curve in $\l^2$ that satisfies equation \eqref{eq-one} is the generating curve of   an $(\alpha+1)$-singular maximal surface invariant by all rotations about the $x$-axis.
\end{proposition}

\begin{example}  \normalfont We know that the solution of \eqref{eq-one} for $\alpha=1$ is the hyperbola $u(x)=\sqrt{1+a^2x^2}/a$, $a>0$. As a consequence of Proposition \ref{pr27}, the only $2$-singular maximal surface that is invariant by the rotations about the $x$-axis is the surface   $x^2+y^2-z^2=-1/a^2$, $z>0$. This surface is the hyperbolic plane $\mathbb{H}^2(1/a)$. Another solution of \eqref{eq-one}   appeared in the Introduction   for $\alpha=-1$. Then the surface generated   is the     hyperbolic catenoid of $\l^3$. 
\end{example}

We finish this section considering   singular maximal surfaces of revolution about a   lightlike axis. Again, we have in mind that if we fix the rotation axis, then the vector $\vec{a}$ in equation \eqref{eqL} is arbitrary. If the rotation axis is determined by the  vector $(1,0,1)$,   the parametrization of the surface is
\begin{equation}\label{eq-li}
X(s,t)=\left(\begin{array}{ccc}1-\frac{t^2}{2}&t&\frac{t^2}{2}\\ -t&1&t\\ -\frac{t^2}{2}&t&1+\frac{t^2}{2}\end{array}\right)\left(\begin{array}{c} u(s)+s\\ 0\\ u(s)-s\end{array}\right),\quad  t\in\r,
\end{equation}
for some function $u=u(s)$, $s\in I\subset\r$.   The spacelike condition on the surface is equivalent to $u'>0$.

\begin{proposition} Let $S$ be a spacelike surface of $\l^3$   invariant by the uniparametric group of rotations about  a   lightlike axis $L$. Suppose $S$ satisfies equation \eqref{eqL} where $\vec{a}$ is a timelike vector. Then  either $\vec{a}$ is orthogonal to  $L$, or $S$ is the hyperbolic plane $\h^2(r)$ being $\vec{a}$ is an arbitrary vector. More precisely, if $L$ is generated by the vector $(1,0,1)$,   $S$ is  parametrized by \eqref{eq-li} and if $\alpha\not=2$, then $\vec{a}=(1,b,1)$, $b\not=0$, and  we have the following possibilities:
\begin{enumerate}
\item If $\alpha=3/2$, then $u(s)=m\log(s)$, $m>0$. 
\item If $\alpha\not= 3/2$, then $u(s)=m s^{3-2\alpha}/(3-2\alpha)$, $m>0$. 
\end{enumerate}
In particular, hyperbolic planes $\h^2(r)$ are   the only  $\alpha$-singular maximal surfaces in $\l^3$ satisfying \eqref{eqL} with $\vec{a}=(0,0,1)$ and  invariant by the group of rotations about  the   lightlike axis generated by the vector $(1,0,1)$.  
\end{proposition} 

\begin{proof}
A straightforward computation of equation   \eqref{eqL} for the surface \eqref{eq-li}  concludes that this equation is a polynomial equation on $t$ of degree $2$. Thus the coefficients corresponding for the variable $t$ must vanish, obtaining
$$2 u' \left((\alpha +1) s (a+c)+(a-c) \left(u+\alpha  s u'\right)\right)-s u'' ((a-c) u+s (a+c))=0$$
$$b  \left(s u''-2 (1-\alpha) u'\right)=0$$
$$ (a-c)\left(s u''-2 (1-\alpha) u'\right)=0.$$
From the second and third equation, if $su''-2(1-\alpha)u'\not=0$, we have $b=0$ and $a=c$, obtaining that $\vec{a}$ is a lightlike vector, which is not possible. Thus $s u''-2 (1-\alpha ) u'=0$. The solution of this equation depends on the value of $\alpha$. 
\begin{enumerate}
\item Case  $\alpha=3/2$. Then $u(s)=m\log(s)$ with $m>0$. The first equation yields $(a-c)m^2(1+\log(s)=0$, that is, $a=c$ and $\vec{a}=(a,b,a)$, $b\not=0$.
\item Case $\alpha\not=3/2$. Then $u(s)=ms^{3-2\alpha}/(3-2\alpha)$ with $m>0$. Now the first equation simplifies into 
$(a-c)(2-\alpha)s^{5-4\alpha}=0$. If $\alpha=2$, then $u(s)=-m/s$ and it is not difficult to see that this surface is   the hyperbolic plane $\mathbb{H}^2(2\sqrt{m})$. If $\alpha\not=2$, then $a=c$, so $\vec{a}=(a,b,a)$, $b\not=0$. 

\end{enumerate}

\end{proof}

 \section{Surfaces of revolution about the $z$-axis}\label{sec3}

In this section we study the   surfaces of revolution with timelike axis  $L$. Again, the same observations done in the previous section hold in the sense that   there is not an {\it a priori} relation between the vector $\vec{a}$ and the axis $L$. The first result that we will prove is that, indeed,  $L$ must parallel to the vector $\vec{a}$.

\begin{proposition}\label{pr-zz}
 Let $S$ be an $\alpha$-singular maximal surface in $\l^3$ that is invariant by the uniparametric group of rotations about a timelike axis $L$. Suppose that $S$ satisfies equation \eqref{eqL} where  $\vec{a}$ is now an arbitrary timelike vector. Then either $L$ and $\vec{a}$ are parallel, or $S$ is the hyperbolic plane $\h^2(m)$ being $\vec{a}$ is an arbitrary timelike vector.
\end{proposition}

\begin{proof} After a rigid motion, we suppose that the rotation axis is the $z$-axis. Let  $\vec{a}=(a,b,c)$  after this motion. The surface $S$ parametrizes  as $X(r,\theta)=(r\cos\theta,r\sin\theta,u(r))$, $r\in I\subset\r^+$, $\theta\in\r$, $u>0$ and $u'^2<1$. The computation of equation \eqref{eqL} gives 
a polynomial equation on  the trigonometric functions $\{1,\sin\theta,\cos\theta\}$. Thus all three coefficients must vanish, obtaining
$$a  \left(r u''+(\alpha +1) u'(1-u'^2)\right)=0$$
 $$b \left(r u''+(\alpha +1) u'(1-u'^2)\right)=0$$
$$c \left(\alpha  r \left(1-u'^2\right)+u \left(r u''+u'(1-u'^2)\right)\right)=0.$$
If $r u''-(\alpha +1) u'(1-u'^2)\not=0$, then $a=b=0$, proving that $\vec{a}=(0,0,c)$, hence $L$ and $\vec{a}$ are parallel.

Suppose now that $r u''+(\alpha +1) u'(1-u'^2)=0$.  Recall that $c\not=0$ because $\vec{a}$ is a timelike vector. Combining with the third equation, we find $uu'-r=0$. Solving this equation we obtain 
$u(r)=\sqrt{r^2+m^2}$, $m>0$, and the corresponding surface   is the hyperbolic plane $\h^2(m)$.  
\end{proof}

By Proposition \ref{pr-zz}, and after a horizontal translation, we will assume that the rotation axis is the $z$-axis  and     $\vec{a}=(0,0,1)$ in \eqref{eqL}. We know that   $X(r,\theta)=(r\cos\theta,r\sin\theta,u(r))$, where $r\in I\subset \r^+$, $\theta\in\r$ and $u>0$. By the proof of     Proposition \ref{pr-zz},   equation (\ref{eqL}) writes as  
 \begin{equation}\label{eq3}
 \frac{u''}{(1-u'^2)^{3/2}}+\frac{u'}{r\sqrt{1-u'^2}}=\frac{\alpha}{u\sqrt{1-u'^2}},
 \end{equation}
 or equivalently,
 \begin{equation}\label{eq33}
  \frac{u''}{ 1-u'^2 }+\frac{u'}{r }=\frac{\alpha}{u}.
 \end{equation}
 We are interested in those solutions that meet the $z$-axis, that is, when $r=0$ is contained in the domain of the solution. Let us observe that    equation     (\ref{eq3}) is singular at $r=0$ and thus the existence of solutions   is not a direct consequence of standard ODE theory. 
 
  Multiplying   (\ref{eq3}) by $r$, and integration by parts, we wish to establish the existence of a classical solution of
\begin{equation}\label{rot}
\left\{\begin{array}{ll}
  \left(r\dfrac{u'}{\sqrt{1-u'^2}}\right)'=r\dfrac{\alpha}{u\sqrt{1-u'^2}}&\mbox{   $r\in (0,\delta)$}\\
u(0)=u_0>0, \quad u'(0)=0.&
\end{array}\right.
\end{equation}

 Define the functions $\phi:(-1,1)\rightarrow\r$ and $f:\r^+\times(-1,1)\rightarrow\r$ by
 $$\phi(y)=\frac{y}{\sqrt{1-y^2}}\quad f(x,y)=\frac{\alpha}{x\sqrt{1-y^2}}.$$
 
Let $\delta>0$. It is clear that a function $u\in C^2([0,\delta])$  is a solution of (\ref{rot}) if and only if   $(r\phi(u'))'=r f(u,u')$ and $u(0)=u_0$, $u'(0)=0$. Let  $\mathcal{B}=(C^1([0,\delta]),\|\cdot\|)$ be the Banach space  of the continuously differentiable functions on $[0,\delta]$ endowed with the  usual norm 
 $$\|u\|=\|u\|_{\infty}+\|u'\|_{\infty}.$$
Define the operator ${\mathsf T}:\mathcal{B}\rightarrow \mathcal{B}$ by
$$({\mathsf T}u)(r)=u_0+\int_0^r\phi^{-1}\left(\int_0^s\frac{t}{s} f(u,u') dt\right)ds.$$
Notice that a fixed point of the operator ${\mathsf T}$ is a solution of the initial value problem (\ref{rot}). Indeed, $({\mathsf T}u)'=\phi^{-1}\left(\frac{1}{r}\int_0^r t f(u,u')dt\right)$ and 
$$r\phi({\mathsf T} u')\int_0^r t f(u,u')dt,$$
 obtaining the result. Moreover, ${\mathsf T}u(0)=u_0$ and 
 $$\phi({\mathsf T}u)'(0)=\lim_{r\rightarrow 0}\frac{1}{r}\int_0^r t f(u,u')dt=\lim_{r\rightarrow 0} r f(u,u')=0,$$
 where in the second identity we have used  the L'H\^{o}pital rule. Thus, $({\mathsf T}u)'(0)=0$.

The existence of solutions of \eqref{rot} follows  now standard techniques of radial solutions for some   equations of mean curvature type (\cite{be,cco}). In Figure \ref{fig2} we show the solutions of \eqref{rot} when $\alpha$ is positive and negative.
 
 \begin{figure}[hbtp]
\centering \includegraphics[width=.4\textwidth]{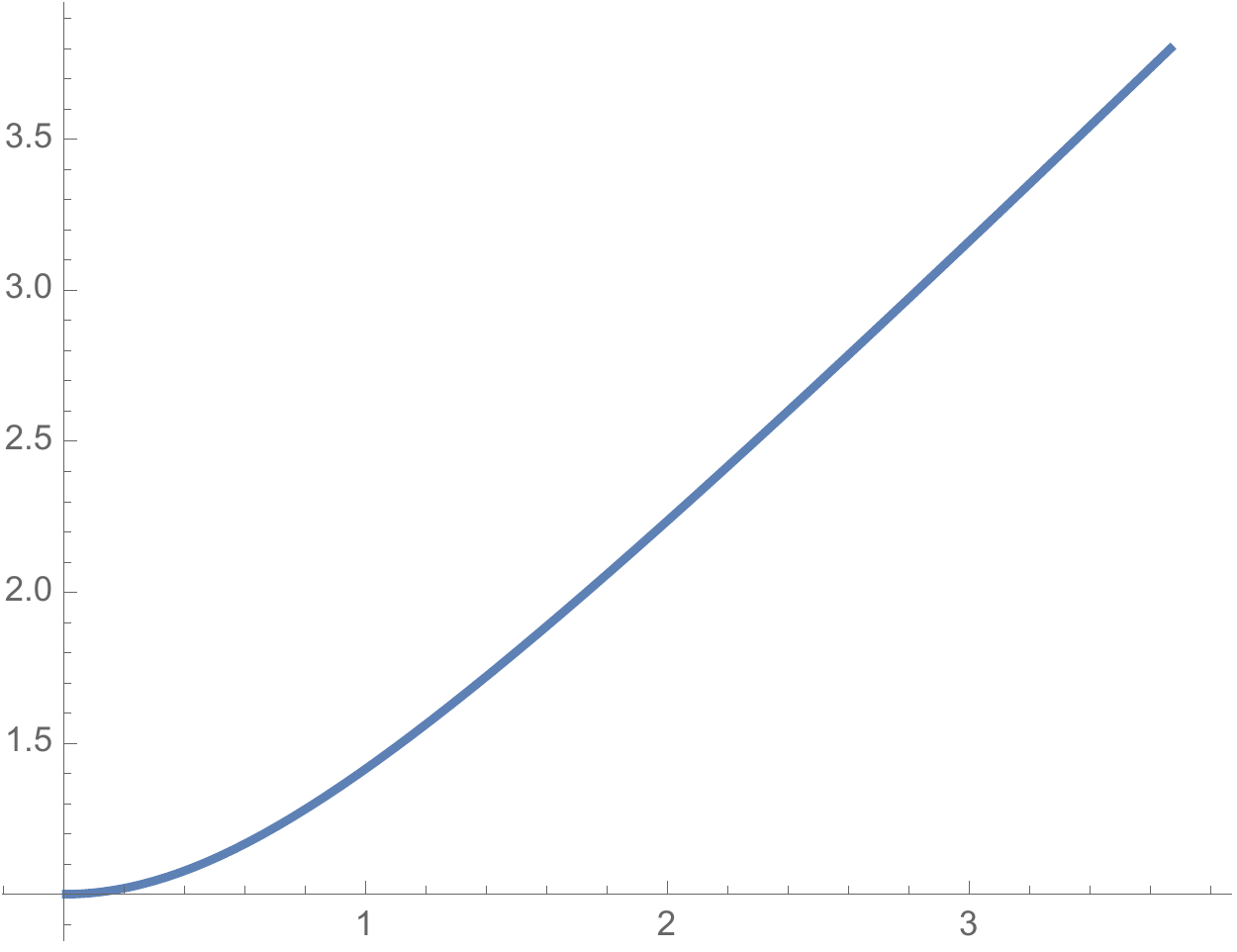}\quad \includegraphics[width=.4\textwidth]{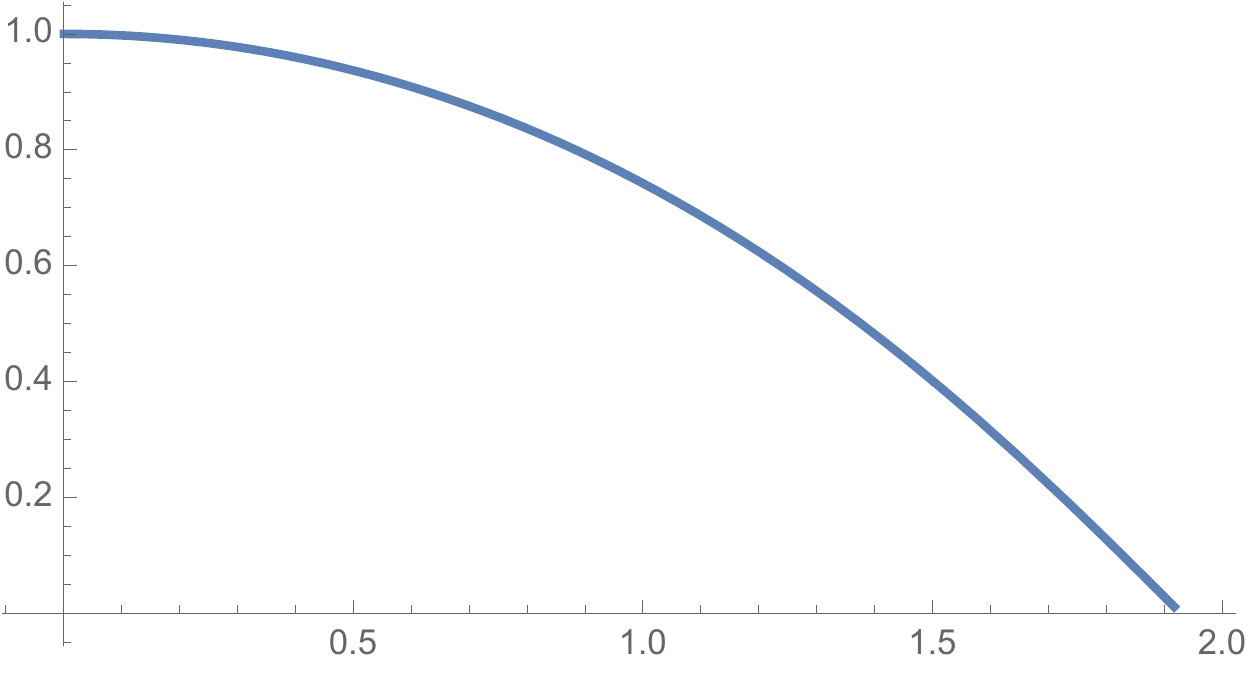}
 \caption{Solutions of \eqref{rot}. Left: case $\alpha>0$, here $\alpha=2$. Right: case $\alpha<0$, here $\alpha=-1$ }\label{fig2}
 \end{figure}

\begin{theorem}\label{pr-exi}
The initial value problem (\ref{rot})  has a solution $u\in C^2([0,\delta])$ for some $\delta>0$ that depends continuously on the initial data.
\end{theorem}

\begin{proof}

In order to find a fixed point of ${\mathsf T}$, we prove that   ${\mathsf T}$ is a contraction in $\mathcal{B}$ for some $\delta>0$ to be chosen. The functions $f$ and  $\phi^{-1}$ are  locally Lipschitz continuous of constant $L>1$ in
$[u_0-\epsilon,u_0+\epsilon]\times[-\epsilon,\epsilon]$ and $[-\epsilon,\epsilon]$ respectively, provided $\epsilon<\{u_0,1\}$. Since  $\phi^{-1}(y)=y/\sqrt{1+y^2}$, then $L<1$.  Then for all $u,v\in\overline{B(0,\epsilon)}$ and for all $r\in [0,\delta]$,
\begin{eqnarray*}
|({\mathsf T}u)(r)-({\mathsf T}v)(r)|&\leq& L\int_0^r\left|\int_0^s\frac{t}{s}(f(u,u')-f(v,v'))dt\right|\\
&\leq& L^2\int_0^r\int_0^s\frac{t}{s}\|u-v\| dt=
\frac{L^2}{4} r^2  \|u-v\|.
\end{eqnarray*}
\begin{eqnarray*}
|({\mathsf T}u)'(r)-({\mathsf T}v)'(r)|&\leq&\frac{L}{r}\left|\int_0^r t(f(u,u')-f(v,v'))dt\right|\\
&\leq&\frac{L^2}{r}\int_0^r t\|u-v\|dt=\frac{L^2}{2}r\|u-v\|.
\end{eqnarray*}
By choosing $\delta>0$ small enough, we deduce that ${\mathsf T}$ is a contraction in the closed ball $\overline{B(0,\delta)}\subset \mathcal{B}$. Thus the Schauder Point Fixed Theorem proves the existence of one fixed point of $\mathsf{T}$, so  the existence of  a local solution of the initial value problem  (\ref{rot}). This solution belongs to $C^1([0,\delta])\cap C^2((0,\delta])$. The $C^2$-regularity up to $0$ is verified directly by  using the L'H\^{o}pital rule because (\ref{eq3})  leads to
     $$\lim_{r\rightarrow 0}u''(r)+\lim_{r\rightarrow 0}\frac{u'(r)}{r}=\frac{\alpha}{u_0},$$
that is,
\begin{equation}\label{uu}
\lim_{r\rightarrow 0} u''(r)=\frac{\alpha}{2u_0}.
\end{equation}
The continuous dependence of local solutions on the initial data  is a consequence of the continuous dependence of the fixed points of ${\mathsf T}$.
 \end{proof}

In the following result we describe the geometric properties of the rotational solutions of \eqref{eq33}. See figures \ref{fig2}, \ref{fig3} and \ref{fig4}.
 
\begin{theorem}\label{t32}
 Let $u$ be a solution of \eqref{eq3} with $u>0$ and $u'^2<1$.
  \begin{enumerate}
 \item Case $\alpha>0$.  The maximal domain of $u$ is $(0,\infty)$. Let 
 $u_0'=\lim_{r\rightarrow 0}u'(r)$.  Then  we have the following cases: 
  $u_0'=0$ and the function $u$ is increasing; $u_0'=-1$ and   $u$ has a unique critical point which is a global minimum; $u_0=1$ and the function is increasing.  In all cases, 
\begin{equation}\label{t32-1}
\lim_{r\rightarrow\infty}u(r)=\infty.
\end{equation}
Also, the function  $u(r)=\sqrt{\alpha} r$ is a solution of  (\ref{eq3}).

 \item Case $\alpha<0$.  The maximal domain of $u$ is $(a,b)$ with  $0\leq a<b<\infty$ and 
 $$\lim_{r\rightarrow b}u(r)=0,\quad \lim_{r\rightarrow b}u'(r)=-1.$$
If $a>0$, then $u$ has a global maximum and 
$$\lim_{r\rightarrow a}u(r)=0,\quad\lim_{r\rightarrow a}u'(r)=1.$$
If $a=0$, let $u_0'=\lim_{r\rightarrow 0}u'(r)$. Then we have the following cases:   $u_0'=0$ and $u$ is a decreasing function;  $u_0'=-1$ and    $u$ is a decreasing function; $u_0'=1$ and $u$ has a   global maximum. 
\end{enumerate}

\end{theorem}

\begin{proof}We observe that if $u$ has a critical point at   $r_o\geq 0$, then \eqref{eq33} implies $u''(r_o)=\alpha/u(r_o)\not=0$, hence all critical points are all maximum or are all minimum.  Thus there is one critical point at most. In such a case, this point is a global minimum (resp. maximum) if $\alpha>0$ (resp. $\alpha<0$). 

{\it Claim A. If the graphic of $u$ meets the $x$-axis at $r_*>0$, then $\alpha<0$ and $\lim_{r\rightarrow r_*}u'(r)^2=1$.} 

The proof follows by multiplying  \eqref{eq33} by $2u'$ and integrating. Then
$$\log(1-u'(r)^2)+2\alpha\log u(r)=2\alpha\int^r\frac{u'(t)^2}{t}dt+\mu,\quad\mu\in\r.$$
In a neighborhood of $r_*$, the right-hand side of the above equation is finite. Since $\log(u(r))\rightarrow -\infty$ as $r\rightarrow r_*$, the same occurs with $\lim_{r\rightarrow r_*}\log(1-u'(r)^2)$, proving that $u'(r)^2\rightarrow 1$ as $r\rightarrow r_*$. Moreover, the case $\alpha>0$ is not possible because the left-hand side would be $-\infty$. 

{\it Claim B. If the graphic of $u$ meets the $y$-axis, then $\lim_{r\rightarrow 0}u'(r)=0$ or $\lim_{r\rightarrow 0}u'(r)^2=1$.} 

Let denote $u_0'=\lim_{r\rightarrow 0}u'(r)$. From Theorem \ref{pr-exi}, we know  the existence of solutions when $u_0'=0$. Suppose now $u_0'\not=0$. By contradiction, we assume that $u_0'^2\not=1$. For $\delta>0$ close to $0$ and by  \eqref{rot}, 
\begin{equation}\label{delta}
\frac{ru'(r)}{\sqrt{1-u'(r)^2}}-\frac{\delta u'(\delta)}{\sqrt{1-u'(\delta)^2}}=\int_\delta^r\frac{\alpha t}{u\sqrt{1-u'^2}}dt.
\end{equation}
Since $u_0'^2\not=1$, letting $r\rightarrow 0$ we have
$$\frac{u'(\delta)}{\sqrt{1-u'(\delta)^2}}=\frac{1}{\delta}\int_0^\delta\frac{\alpha t}{u\sqrt{1-u'^2}}dt.$$
Letting $\delta\rightarrow 0$ and by the L'H\^{o}pital rule, we deduce
$$\lim_{\delta\rightarrow 0}\frac{u'(\delta)}{\sqrt{1-u'(\delta)^2}}=\lim_{\delta\rightarrow 0}\frac{\alpha\delta}{u\sqrt{1-u'^2}}=0,$$
hence $u_0'=0$, a contradiction. 

In particular, the claim B implies that it is not possible to find   solutions of the initial value problem   \eqref{rot} when $u_0'^2\in (0,1)$.

From now, we will denote by $u(a)$ and $u'(a)$  (similar for $r=b$), the limit of $u(r)$ and $u'(r)$ at $r=a$.

{\it Claim C. If $a>0$ (resp. $b<\infty$), then $u(a)=$ (resp. $u(b)=0$).} 

Suppose that $a>0$ (similarly for $b<\infty$). If $u(a)\not=0$, then $u''$ is bounded around $r=a$ by \eqref{eq33}. Since $u'$ and $u''$ are bounded functions, we could extend the solution $u$ beyond $r=a$, a contradiction.

We are in position to prove the theorem. 

\begin{enumerate}
 
\item Case $\alpha>0$. Suppose that  $u'>0$ in all its domain. Since $u'$ and $u''$ are bounded functions  by \eqref{eq33}, then the value of $b$ in $I$ is $b=\infty$. If $a>0$, this implies that $u(a)=0$ by Claim C and  this a contradiction by   Claim A. This proves that $I=(0,\infty)$.

Suppose that the sign of $u'$ is negative in all its domain.  Then \eqref{eq33} implies that $u$ is a concave function and thus $b<\infty$   because $u$ is decreasing. Then $u(b)=0$, which is not possible by  Claim A.

After the above arguments, we have proved that if $u'$ has a constant sign, then $u'>0$, $a=0$ and either $u_0'=0$ or $u_0'=1$. In case that $u'$ changes of sign, then there is a unique critical point at some point $r=r_o>0$, which is a global minimum. In this case, $u'>0$ for $r>r_o$. Since $u'$ and $u''$ are bounded, then   $b=\infty$. The case $a>0$ is forbidden by Claim C.   Thus $a=0$. Since $u'<0$ for $r<r_o$,   Claim B asserts  $u_0'=-1$.

We prove  \eqref{t32-1}. Since $u$ is increasing  close $\infty$, let  $c=\lim_{r\rightarrow\infty}u(r)$. If $c<\infty$, then $u'(r)\rightarrow 0$ as $r\rightarrow\infty$ and using \eqref{eq33}, $\lim_{r\rightarrow\infty}u''(r)= \alpha/c>0$, a contradiction. Thus $c=\infty$. 

Finally, by a direct computation, we observe that  $u(r)=\sqrt{\alpha}r$ is a solution of  (\ref{eq3}).

\item Case $\alpha<0$. Suppose that $u'<0$ in all its domain. Let $c=\lim_{r\rightarrow b}u(r)\geq 0$. If $b=\infty$, then  $u'(r)\rightarrow 0$ and  \eqref{eq33} would imply that $\lim_{r\rightarrow \infty}u''(r)$ is either $\alpha/c$ if $c>0$ or $\infty$ if $c=0$, a contradiction. Thus $b<\infty$, hence $u(b)=0$. By Claim A, $u'(b)=-1$. If $a>0$, then  $u(a)>0$ because $u$ is decreasing: a contradiction by Claim C.  Thus $a=0$. By Claim B and because $u$ is decreasing, we have two possibilities, namely,  $u_0'=0$ and $u_0'=-1$.

Suppose now that $u'>0$ in all its domain. Then $u$ is a concave function by \eqref{eq33}.    By Claim C, we have $b=\infty$. In  the other end of the interval $I$, namely $r=a$, we have $a=0$, $u_0'=1$ or $a>0$, $u(a)=0$ and $u'(a)=1$. In both cases, as $u''<0$, we find  $\lim_{r\rightarrow\infty} u'(r)=\lim_{r\rightarrow\infty} u''(r)=0$ and  $\lim_{r\rightarrow\infty} u(r)=\infty$. By \eqref{delta}
$$\frac{u'(r)}{\sqrt{1-u'(r)^2}}=\frac{1}{r}\int_\delta^r\frac{\alpha t}{u\sqrt{1-u'^2}}dt+\mu.$$
Letting $r\rightarrow\infty$,  the left-hand side is $0$. However, and applying  twice the   L'H\^{o}pital rule, the limit of the  right-hand side is 
$$\lim_{r\rightarrow\infty}\frac{\alpha r}{u(r)}+\mu =\lim_{r\rightarrow\infty}\frac{\alpha}{ u'(r)}+\mu=\infty,$$
obtaining a contradiction.

Thus, if $u'>0$ at some point, there is a critical point $r_o$ of $u$, which will be the global maximum of $u$.  Then  $a\geq 0$ with  $u'(a)=1$ because $u$ is increasing in $(a,r_o)$.  
\end{enumerate}

\end{proof}
 
 \begin{remark} If $\alpha<0$, there exist solutions that do not meet the rotation axis, see figure \ref{fig4}, right. This case appears if $0<a<b<\infty$, where the function $u$ has a global maximum and $u'(a)=1=-u'(b)$. This extends the same property of the solution of \eqref{eq0}, where the   part of the   function $u=u(x)$ given in \eqref{eq00} that lies over   the $x$-axis is formed by successive  bounded intervals.  
 \end{remark}
 
  \begin{figure}[hbtp]
 \centering \includegraphics[width=.4\textwidth]{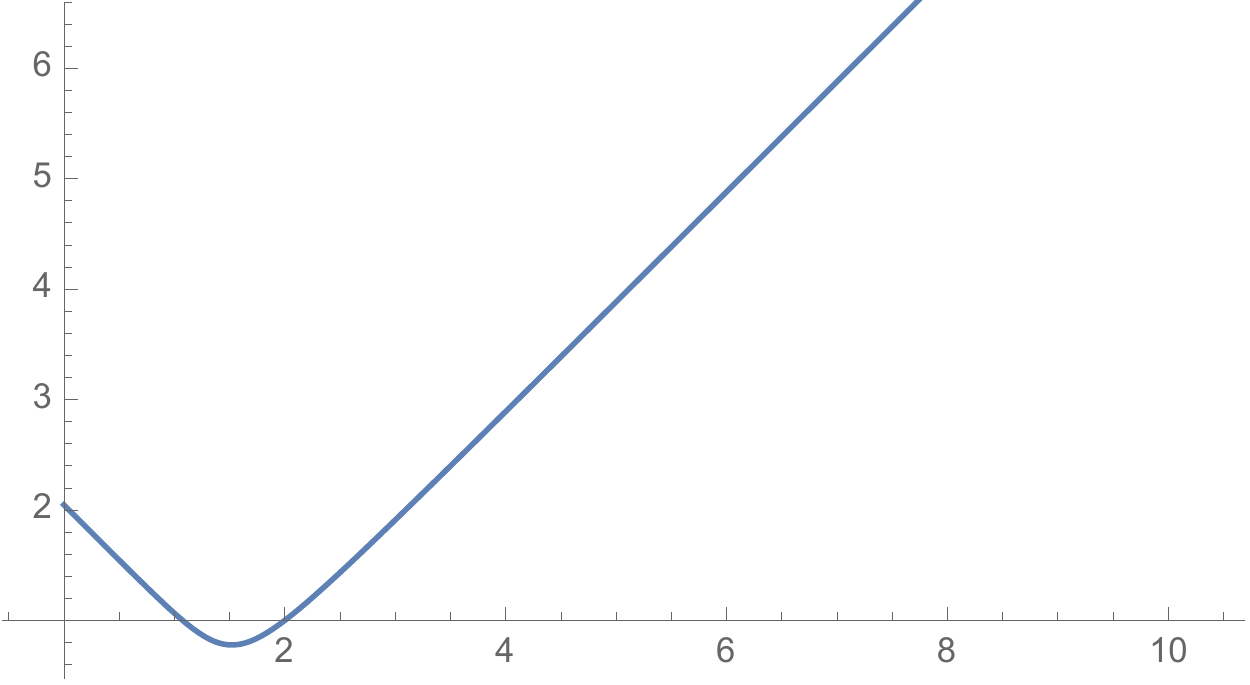}\quad \includegraphics[width=.3\textwidth]{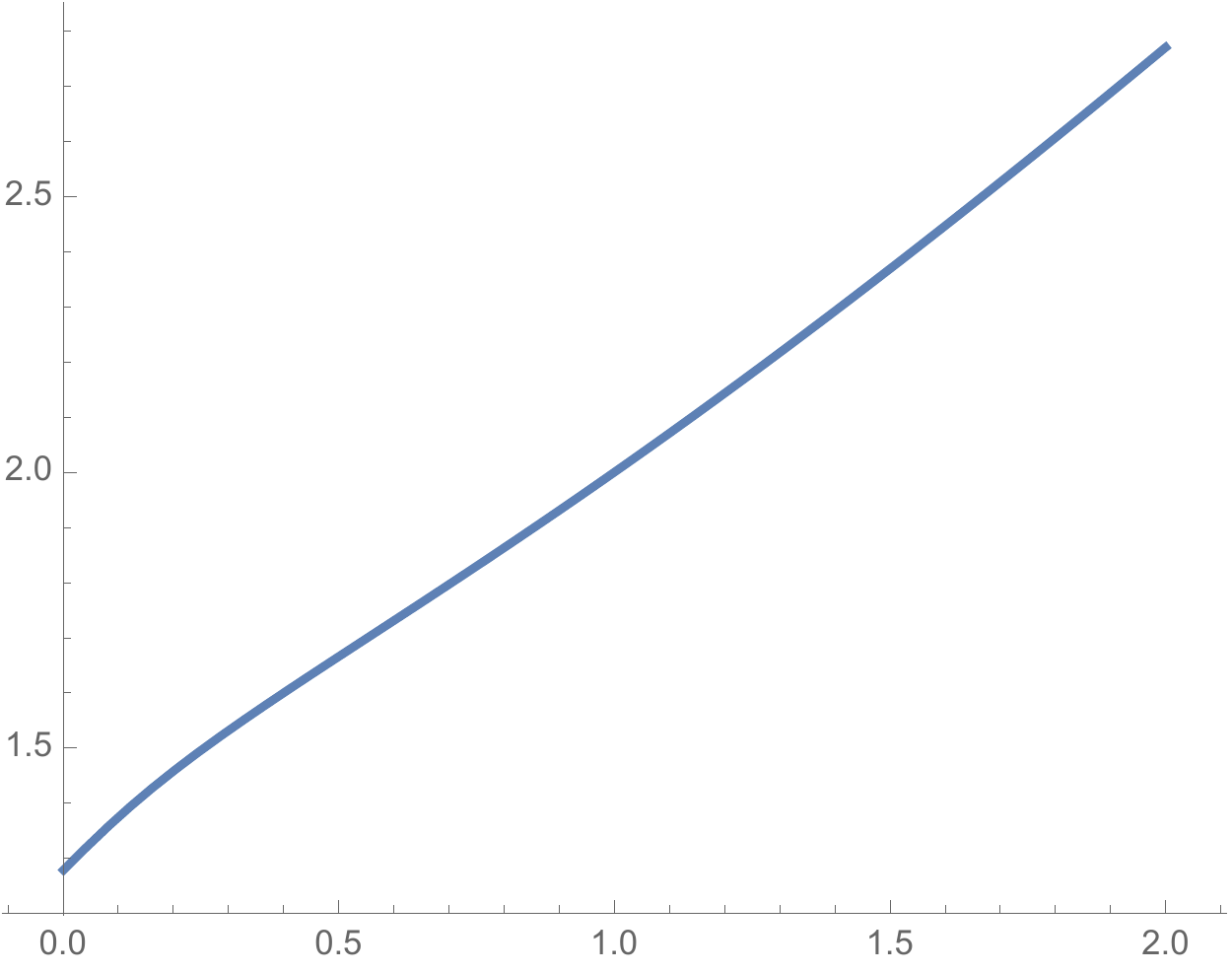}
 \caption{Solutions of \eqref{eq33}, case $\alpha>0$ and $u_0'\not=0$. Left: case $u_0'=-1$. Right: case $u'_0=1$ }
 \end{figure}\label{fig3}
 
  \begin{figure}[h]
 \centering \includegraphics[width=.32\textwidth]{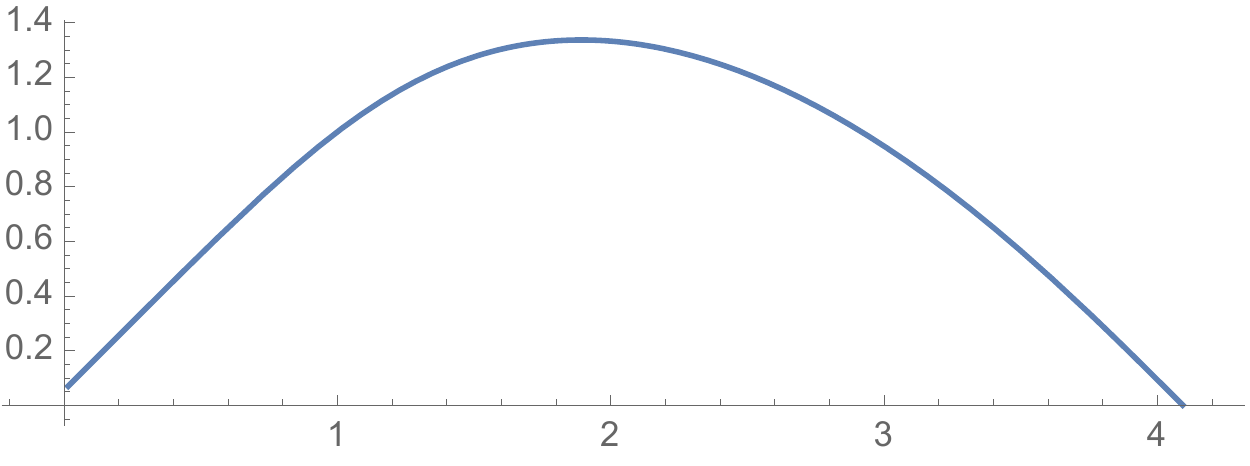}\quad \includegraphics[width=.25\textwidth]{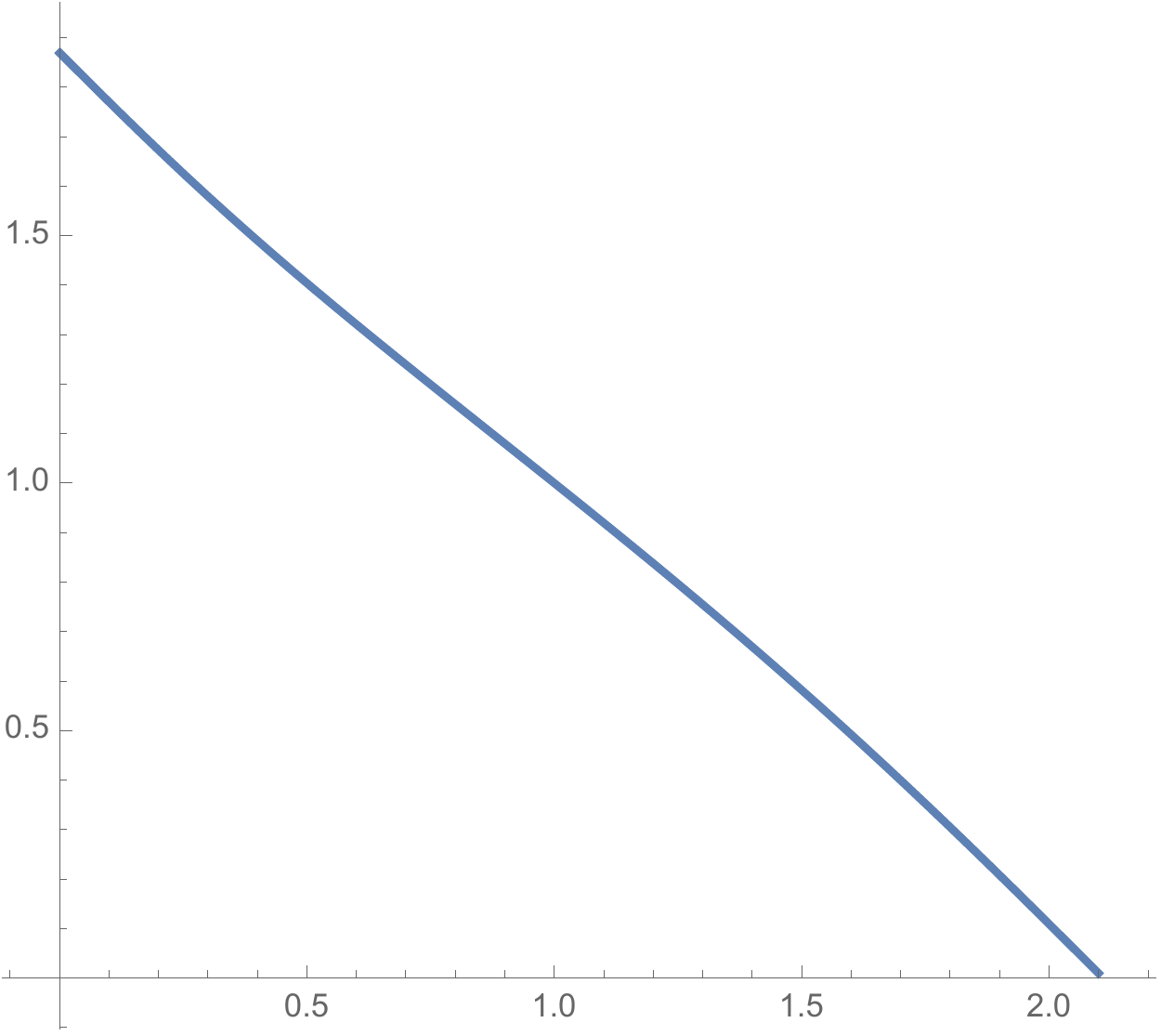}\quad \includegraphics[width=.32\textwidth]{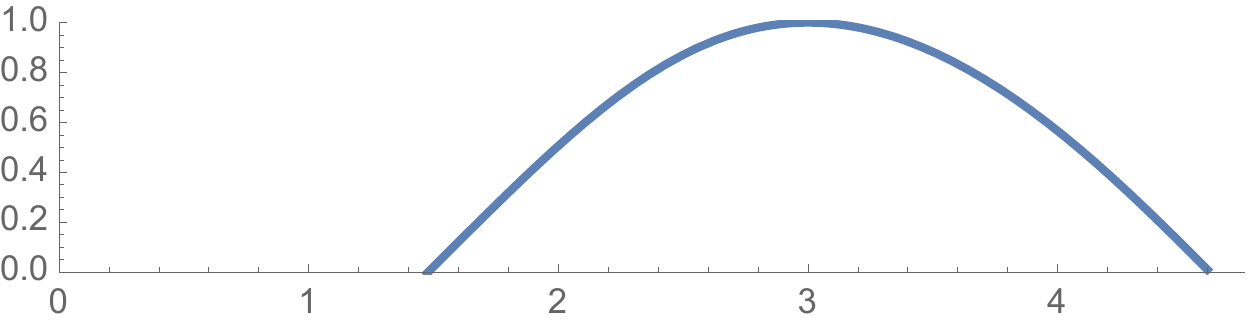}
 \caption{Solutions of \eqref{eq33}, case $\alpha<0$ and $u_0'\not=0$. Left: case $u_0'=1$. Middle: case $u'_0=-1$. Right: a solution that does not meet the rotation axis }\label{fig4}
 \end{figure}

 \section{The Dirichlet problem}\label{sec4}

The Dirichlet problem of the singular maximal surface equation asks if given a positive function $\varphi:\partial\Omega\rightarrow\r$ defined in a bounded domain $\Omega\subset\r^2$,   there exists a smooth positive function $u:\overline{\Omega}\rightarrow\r$ such that \eqref{eq2} holds in $\Omega$, $u=\varphi$ on $\partial\Omega$ and $|Du|<1$ on $\overline{\Omega}$. Since any curve in a spacelike surface must be spacelike,   the graph $\Gamma$ of $\varphi$ is   spacelike. The problem is to determine the type of function $\varphi$ and the boundary $\partial\Omega$ for the solvability of the Dirichlet problem. It is expectable that  the sign $\alpha$ in \eqref{eq2} plays an important role because we have seen in Sections \ref{sec2} and \ref{sec3} the contrast of the behaviour of the invariant solutions of \eqref{eqL} depending if $\alpha$ is positive or negative.

Following similar ideas of Jenkins and Serrin  in \cite{js,se}, we will solve the Dirichlet problem if the domain $\Omega$ is mean convex. In fact,   we will establish the Dirichlet problem in the $n$-dimensional case, or equivalently, we will find singular maximal hypersurfaces in the $(n+1)$-dimensional Lorentz-Minkowski space $\l^{n+1}$ with prescribed boundary data. 

 Recall that a bounded domain $\Omega\subset\r^n$ is said to be mean convex if $\partial\Omega$ has nonnegative mean curvature $H_{\partial\Omega}$ with respect to the inward orientation. In case $n=2$, the mean convexity property is equivalent to the convexity of $\Omega$, but in arbitrary dimensions, the mean convexity is less restrictive than convexity. 
 
The Dirichlet problem is now formulated as follows. Let  $\Omega\subset\r^n$ be a smooth  bounded  domain and $\alpha\not=0$ a given constant.  Let  $\varphi:\partial\Omega\rightarrow\r$ be a positive spacelike smooth function. The problem is finding  a   classical  solution $u\in C^2(\Omega)\cap C^0(\overline{\Omega})$, $u>0$ in $\overline{\Omega}$,   of

\begin{equation}\label{d1}
\left\{\begin{split}
&\mbox{div}\left(\dfrac{Du}{\sqrt{1-|Du|^2}}\right)= \frac{\alpha}{u\sqrt{1-|Du|^2}}\quad \mbox{in $\Omega$}\\
&u=\varphi\quad \mbox{on $\partial\Omega$}\\
&|Du|<1\quad \mbox{in $\overline{\Omega}.$}
\end{split}\right.
\end{equation}

We solve the Dirichlet problem when the boundary data $\varphi$  has a spacelike extension in $\overline{\Omega}$.

\begin{theorem} \label{t1}
Let $\Omega\subset\r^n$ be a  bounded mean convex  domain with smooth boundary $\partial\Omega$. Assume that $\alpha<0$. If $\varphi\in C^{2}(\overline{\Omega})$   is a positive function  with $\max_{\overline{\Omega}}|D\varphi|<1$,   then there is a unique positive solution $u$ of \eqref{d1}.
\end{theorem}

The proof of Theorem \ref{t1} is accomplished by using   the Schauder theory of {\it a priori}  global estimates, the method of continuity and the Leray-Schauder fixed point theorem. Applying these techniques, we find all elements   for proving Theorem \ref{t1}. As usual,   we will utilize  the distance function $d$ to $\partial\Omega$ to construct a barrier function (\cite{gt,js,lu}).

The $C^0$ estimates will be obtained by comparing the solution of \eqref{d1} with the rotational examples studied in Section \ref{sec3}: here the hypothesis $\alpha<0$ will be essential because if    $\alpha>0$  it is not possible   to prevent that $|u|\rightarrow 0$  for a solution $u$. For the $C^1$ estimates, we need to prove that $|Du|$ is bounded away from $1$ which will be deduced   by using  barrier functions. Finally, the hypothesis $\alpha<0$ will be also used when we apply the Implicit  Function Theorem for the existence of the linearized problem associated to \eqref{d1}.

The maximum principle for elliptic equations of divergence type implies  the following result.  
 
  \begin{proposition}[Touching principle]\label{pr21} Let $\Sigma_1$ and $\Sigma_2$ be two $\alpha$-singular maximal surfaces.   If   $\Sigma_1$ and  $\Sigma_2$ have a common tangent interior point   and $\Sigma_1$ lies above $\Sigma_2$ around $p$, then $\Sigma_1$ and $\Sigma_2$ coincide at an open set around $p$. 
\end{proposition}

We also need to formulate the   comparison principle in the context of $\alpha$-singular maximal surfaces. We write the equation of \eqref{d1} in classical notation.  Define the operator 
\begin{equation}\label{op}
\begin{split}
Q[u]&= (1-|Du|^2)\Delta u+u_iu_ju_{ij}-\frac{\alpha(1-|Du|^2)}{u}\\
&=a_{ij}(Du)u_{ij}+{\textbf b}(u,Du),
\end{split}
\end{equation}
 where 
 $$a_{ij}=(1-|Du|^2)\delta_{ij}+u_iu_j,\quad {\textbf b}= - \frac{\alpha(1-|Du|^2)}{u}.$$
 Here   $u_i=\partial u/\partial x_i$, $1\leq i\leq n$,  and we assume the summation convention of repeated indices.    It is immediate that $u$ is a solution of equation (\ref{d1}) if and only if $Q[u]=0$.   The ellipticity of  the operator $Q$ is clear  because if $A=(a_{ij})$ and $\xi\in\r^n$, then 
\begin{equation}\label{px}
(1-|p|^2)|\xi|^2\leq \xi^tA\xi=(1-|p|^2)|\xi|^2+\langle p,\xi\rangle^2\leq|\xi|^2.
\end{equation}
Moreover, this   shows that $Q$  is not uniformly elliptic.  We recall the comparison principle (\cite[Th. 10.1]{gt}).

\begin{proposition}[Comparison principle] \label{pr-43}
Let $\Omega\subset\r^n$ be a bounded domain. If $u,v\in C^2(\Omega)\cap C^0(\overline{\Omega})$ satisfy $Q[u]\geq Q[v]$ and $u\leq v$ on $\partial\Omega$, then $u\leq v$ in $\Omega$.
\end{proposition}
 
Notice that if $\alpha<0$, the classical theory implies the uniqueness of solutions of the Dirichlet problem.
 
\begin{proposition}\label{pr-u}
Let $\Omega\subset\r^n$ be a bounded domain and $\alpha<0$.  The solution of \eqref{d1}, if exists, is unique.
  \end{proposition}   
 
In arbitrary dimension, it holds the property that any horizontal translation and any dilation from a point of $\r^n\times\{0\}$ preserves equation \eqref{d1}. Similarly, Theorem \ref{t32}  holds where now  \eqref{rot} is 
$$  \left(r\dfrac{u'}{\sqrt{1-u'^2}}\right)'=r^{n-1}\dfrac{\alpha}{u\sqrt{1-u'^2}}.$$
 We establish the solvability of \eqref{d1} in the particular case that $\Omega$ is a   ball of $\r^n$ and $\varphi$ is a positive constant. 
 
 \begin{proposition}\label{pr25} Let $\alpha<0$ and   $B_R\subset\r^n$ be a round ball of radius $R>0$. If $c>0$, then there is a unique radial  solution $u$ of  (\ref{d1}) in $B_R$ with  $u=c$ on $\partial B_R$.
 \end{proposition}

\begin{proof}
After a horizontal translation, we suppose that the origin $O\in\r^n$ is the center of $B_R$. By Proposition \ref{pr-exi}, let $v=v(r)$ be the solution of  \eqref{rot} with $v(0)=1$. Recall that Theorem \ref{t32}  asserts  that the maximal domain of $v$ is a ball $B_b$ for some $b>0$ with $v(b)=0$.  In the $(r,v)$-plane, consider the line $x_{n+1}=cr/R$. Since $v$ is  a decreasing function, the graph of $v$ meets this line at one point $r=r_o$, $u(r_o)=cr_o/R$. If $\lambda=R/r_o$, then $u_\lambda(r)=\lambda u(r/\lambda)$ is a solution of \eqref{d1} with $u_\lambda(R)=c$. 
\end{proof}

 Following a standard scheme, we start by finding $C^0$ estimates by using   the rotational solutions of \eqref{eqL}.   In the following result, we do not require the mean convexity of  $\Omega$. 
 
  \begin{proposition}  \label{pr-31}
   Let $\Omega\subset\r^n$ be a bounded domain and   $\alpha<0$. If  $u$ is a positive solution of \eqref{d1},   there exists a constant $C_1=C_1(\alpha,\Omega,\varphi)>0$  such that 
\begin{equation}\label{eh}
 \min_{\partial\Omega}\varphi\leq u\leq C_1 \quad \mbox{in $\Omega$}.
\end{equation} 
\end{proposition} 
  
  \begin{proof} Since   the right-hand side of   (\ref{d1}) is  negative, then $\inf_\Omega u=\min_{\partial\Omega}\varphi$ by the maximum principle. This proves the left inequality of \eqref{eh}.

 For the upper estimate of \eqref{eh}, we consider the radial solution $v$ of \eqref{rot} with $v(0)=1$ and let $\{v_\lambda:\lambda>0\}$ where $v_\lambda(r)=\lambda v(r/\lambda)$. Denote $B_R$ the maximal domain of $v$, with $v(R)=0$ and let $\Sigma_\lambda$ denote  the graph of $v_\lambda$. Take $\lambda>0$ sufficiently big so the graph $S$ of $u$ is included in the domain of the halfspace $x_{n+1}>0$ bounded by $\Sigma_\lambda\cup B_{\lambda R}$.  Let    $\lambda$ decrease to $0$ until the first time $\lambda_0$ that $\Sigma_\lambda$ meets $\Sigma_u$. By the maximum principle, the first  contact   must occur at some boundary point of $S$. Then this point is a point of $\partial\Omega$ or a point of $\partial S$. Since $\partial S$ is the graph of $\varphi$, this value $\lambda_0$  depends on $\Omega$ and $\varphi$. Consequently, $u\leq v_{\lambda_0}\leq \sup_\Omega v_{\lambda_0}$. The proof finishes by letting $C_1=\sup_\Omega v_{\lambda_0}$, which  depends only on $\alpha$, $\Omega$ and $\varphi$. 
   \end{proof}
 
 The next step   to prove Theorem \ref{t1} is the derivation of  estimates for $|Du|$. This is done firstly proving the the supremum of $|Du|$ is attained at some boundary point.  In the next result, the assumption that $\alpha$ is negative is essential.

\begin{proposition}[Interior gradient estimates] \label{pr-41} 
   Let $\Omega\subset\r^n$ be a bounded domain and $\alpha<0$. If  $u\in C^2(\Omega)\cap C^1(\overline{\Omega})$ is a positive solution of \eqref{d1}, then 
     $$\max_{\overline{\Omega}}|Du|=\max_{\partial\Omega}|Du|.$$
\end{proposition} 

\begin{proof}
  Let $v^k=u_k$, $1\leq k\leq n$. By differentiating  $Q[u]=0$ with respect to $x_k$, we find for each $k$,
 \begin{equation}\label{eq4}
 \left((1-|Du|^2)\delta_{ij}+u_iu_j\right)v_{ij}^k+2\left(u_i\Delta u+u_ju_{ij}-\frac{\alpha u_i}{u}\right)v_i^k+\frac{\alpha(1-|Du|^2)}{u^2}v^k=0.
 \end{equation}
 Equation (\ref{eq4})  is a linear elliptic equation in  $v^k$. Because $\alpha<0$,    the coefficient for $v^k$ is negative and the maximum principle (\cite[Th. 3.7]{gt}) implies that    $|v^k|$, and then $|Du|$ has not an interior maximum. In particular,   the maximum of $|Du|$ on  $\overline{\Omega}$ is attained at some boundary point,  proving the result.  
 \end{proof}

As a consequence of    Proposition \ref{pr-41},  the problem of  finding    {\it a priori}  estimates of $|Du|$ reduces to get these estimates  along $\partial\Omega$. With this purpose, we prove   that $u$  admits   barriers from above and from below along $\partial\Omega$. It is now when we use the assumption of the mean convexity  of $\Omega$.

     \begin{proposition}[Boundary gradient estimates]  \label{pr42}
   Let $\Omega\subset\r^n$ be a bounded  mean convex   domain and $\alpha<0$. If  $u\in C^2(\Omega)\cap C^1(\overline{\Omega})$ is a positive solution of \eqref{d1}, then there is a constant 
   $$C_2=C_2(\alpha,\Omega, C_1,\|\varphi\|_{1;\overline{\Omega}},\|\varphi\|_{2;\overline{\Omega}})<1$$
    such that 
   $$\max_{\partial\Omega}|Du|\leq C_2.$$
    
\end{proposition}

 \begin{proof} 
 We consider the operator $Q[u]$ defined (\ref{op}). For a lower barrier for $u$, we take  the solution $v^0$ of the Dirichlet problem of the maximal surface equation in $\Omega$ with  the same boundary  $\varphi$. The function $v^0$ is  the solution of \eqref{d1} for $\alpha=0$ whose existen is assured (\cite[Th. 4.1]{bs}). Then
 $$Q[v^0]=-\frac{\alpha(1-|Dv^0|^2)}{v^0}>0=Q[u].$$
 Since $v^0=u$ on $\partial\Omega$, we conclude  $v^0<u$ in $\Omega$ by the comparison principle.

 We now construct an upper barrier  for $u$ by means of  the distance function in a small tubular neighborhood of $\partial\Omega$ in $\Omega$.     
 
 Consider  the distance function $d(x)=\mbox{dist}(x,\partial\Omega)$ and let $\epsilon>0$ sufficiently small so  $\mathcal{N}_\epsilon=\{x\in\overline{\Omega}: d(x)<\epsilon\}$ is a tubular neighborhood of $\partial\Omega$.   We   parametrize $\mathcal{N}_\epsilon$  using normal coordinates $x\equiv (t,\pi(x)) \in\mathcal{N}_\epsilon$, where   $x\equiv\pi(x)+t\nu(\pi(x))$ for some $t\in [0,\epsilon)$,   $\pi:\mathcal{N}_\epsilon\rightarrow\partial\Omega$ is the orthogonal projection and $\nu$ is the unit   normal vector to $\partial\Omega$ pointing to $\Omega$.  A straightforward computation  leads to that   $d$ is   $C^2$,  $|Dd|(x)=1$,  and $\Delta d(x) \leq -(n-1)H_{\partial\Omega}(\pi(x))$  for all $x\in\mathcal{N}_\epsilon$. Because   $\Omega$ is mean convex, then $\Delta d(x)\leq 0$.

Define in $\mathcal{N}_\epsilon$ a function $w=h\circ d+\varphi$, where  we use the same symbol $\varphi$ for a   spacelike extension of $\varphi$  into $\Omega$. The function $h$ is defined as  
\begin{equation}\label{ht}
h(t)=a\log(1+kb^2t),\ b, k >0,\ a=\frac{C_1}{\log(1+b)},
\end{equation}
where $C_1$ is the constant that appears in \eqref{eh} and  $b$ and $k$  will be chosen later. Let us observe that $w>0$ and that we  require that $|Dw|<1$.   The computation of  $Q[w]$ leads to
$$Q[w]=a_{ij}(h''d_id_j+h'd_{ij}+\varphi_{ij})-\frac{\alpha}{w}(1-|Dw|^2).$$
From $|Dd|=1$, it follows that $\langle D(Dd)_x\xi,Dd(x)\rangle=0$ for all $\xi\in\r^n$. If  $\{e_i\}_i$ is the canonical basis of $\r^n$ and $\xi=e_i$, we find $d_{ij}d_j=0$. Thus
\begin{eqnarray*}
w_iw_jd_{ij}&=&(h'd_i+\varphi_i)(h'd_j+\varphi_j)d_{ij}=(h'^2d_i+2h'\varphi_i)d_jd_{ij}+\varphi_i \varphi_jd_{ij}\\
&=&\varphi_{i}\varphi_jd_{ij}\leq \varphi_i^2\lambda_i^d\leq 0,
\end{eqnarray*}
where $\lambda_i^d$ are the eigenvalues of $D^2d$, which all are not positive because $D^2d$ is negative semidefinite. Using this inequality, the definition of $a_{ij}$ in (\ref{op}) and 
\eqref{px}, it follows that 
$$a_{ij}d_{ij}=(1-|Dw|^2)\Delta d+w_iw_j d_{ij}\leq(1-|Dw|^2)\Delta d\leq 0.$$
Again \eqref{px} implies    $a_{ij}d_id_j\geq 1-|Dw|^2$ and $a_{ij}\varphi_{ij}\leq |D^2\varphi|$, where $|D^2\varphi|=\sum_{ij}\sup_{\overline{\Omega}}|\varphi_{ij}|$. Since $h'>0$ and $\Delta d\leq 0$, we find
\begin{equation}\label{qw}
\begin{split}
Q[w]&\leq h''(1-|Dw|^2)+h'\Delta d(1-|Dw|^2)-\frac{\alpha}{w}(1-|Dw|^2)+a_{ij}\varphi_{ij}\\
&\leq  \left(h''-\frac{\alpha}{w} \right)(1-|Dw|^2)+|D^2\varphi|.
\end{split}
\end{equation}
We now study the spacelike condition  $|Dw|<1$. The computation of $|Dw|$ and the Cauchy-Schwarz inequality gives 
$$|Dw|^2=h'^2+|D\varphi|^2+2h'\langle Dd,D\varphi\rangle\leq (h'+|D\varphi|)^2.$$
Because $h'>0$ and $h'$ is decreasing on $t$, we deduce
\begin{equation}\label{grad}
|Dw|\leq h'+|D\varphi|\leq h'(0)+|D\varphi|\leq akb^2+ \mu\quad \mbox{in $\overline{\Omega}$},\end{equation}
where $\mu=\|D\varphi\|_{0;\overline{\Omega}}<1$. Fix a constant $\delta$  with the property   $\mu<\delta<1$. Then it is possible to choose $k$ sufficiently small in \eqref{grad} so $|Dw|\leq akb^2+ \mu<\delta$. Let $\beta=1-\delta^2$. If $h''-\alpha/w<0$, then \eqref{qw} implies 
\begin{equation}\label{qw2}
Q[w]\leq \beta \left(h''-\frac{\alpha}{w} \right)+\|D^2\varphi\|_{0;\overline{\Omega}}.
\end{equation}
The right-hand side in \eqref{qw2} is a function defined in  $\partial\Omega\times [0,\epsilon]$.
Let $\varphi_0=\min_{\overline{\Omega}}\varphi>0$ and we evaluate this function  at $t=0$, obtaining 
$$\beta \left(-ak^2b^4-\frac{\alpha}{\varphi} \right)+\|D^2\varphi\|_{0;\overline{\Omega}}\leq 
\beta \left(-\frac{(\delta-\mu)^2}{a}-\frac{\alpha}{\varphi_0}\right) +\|D^2\varphi\|_{0;\overline{\Omega}}.$$
If $b$ is sufficiently big, then $a\rightarrow 0$, hence the right-hand side in this inequality is   negative. By compactness of $\partial\Omega\times [0,\epsilon]$ and by continuity, let us  take $b$ sufficiently large enough in \eqref{qw2} so $Q[w]<0$. Even more, we require $b$ so large such that  $1/(kb)<\epsilon$. We now change the tubular neighborhood $\mathcal{N}_\epsilon$ by replacing $\epsilon$ by $\epsilon=1/(kb)$ and we denote $\mathcal{N}_\epsilon$ again.

 In order to assure that $w$ is a local upper barrier  in $\mathcal{N}_{\epsilon}$ for the Dirichlet problem \eqref{d1},  we need to have 
\begin{equation}\label{mm}
u\leq w\quad \mbox{in $\partial\mathcal{N}_\epsilon$}.
\end{equation}
 In $\partial\mathcal{N}_\epsilon\cap\partial\Omega$,   the distance function  is $d=0$, so $w=\varphi=u$. On the other hand, in $\partial\mathcal{N}_\epsilon\setminus\partial\Omega$, and  because   $\epsilon=1/(kb)$, we find 
$$w=h(\epsilon)+\varphi=\frac{C_1}{\log(1+b)}\log(1+kb^2\epsilon)+\varphi=C_1+\varphi.$$   
By  Proposition \ref{pr-31}, we have $u\leq C_1$ and   we deduce $u<w$ in $\mathcal{N}_\epsilon\setminus\partial\Omega$.   Definitively, we find $Q[w]<0=Q[u]$ and $u\leq w$ in $\partial\mathcal{N}_\epsilon$, concluding that      $u\leq w$ in $\mathcal{N}_\epsilon$ by the comparison principle. 
  
Consequently, we have proved the existence of lower and upper barriers for $u$  in $\mathcal{N}_\epsilon$, namely, $v^0\leq u\leq w$ in $\mathcal{N}_\epsilon$. Hence we deduce 
$$\max_{\partial\Omega}|Du|\leq C_2:=\max\{\|Dw\|_{0;\overline{\Omega}}, \|Dv^0\|_{0;\overline{\Omega}}\}$$
  and both values $\|Dw\|_{0;\overline{\Omega}}, \|Dv^0\|_{0;\overline{\Omega}}$ depend only on the initial data of the Dirichlet problem.  This completes the proof of proposition. 
    \end{proof}

With all above ingredients, we are in position to prove Theorem \ref{t1}.

\begin{proof}[Proof of Theorem \ref{t1}]

We establish the solvability  of the Dirichlet problem   \eqref{d1} by the method of continuity (see   \cite[Sec. 17.2]{gt}). Define  the family of Dirichlet  problems parametrized by $t\in [0,1]$  
 $$\mathcal{P}_t: \left\{\begin{array}{cll}
Q_t[u]&=&0 \mbox{ in $\Omega$}\\
 u&=& \varphi \mbox{ on $\partial\Omega,$}
 \end{array}\right.$$
 where 
   $$Q_t[u]= (1-|Du|^2)\Delta u+u_iu_ju_{ij}-\frac{\alpha t(1-|Du|^2)}{u}.$$
 
The graph $\Sigma_{u_t}$ of a solution of $u_t$   is a $(t\alpha)$-singular maximal surface.  Notice that if $t=0$, $Q_0[u]=0$ is the maximal surface equation and the solution of $\mathcal{P}_0$ is the function $v^0$ that  appeared in Proposition \ref{pr42}.   As usual, let 
$$\mathcal{A}=\{t\in [0,1]: \exists u_t\in C^{2,\gamma}(\overline{\Omega}), u_t>0, Q_t[u_t]=0, {u_t}_{|\partial\Omega}=\varphi\}.$$ 
The proof consists to show that $1\in \mathcal{A}$. For this, we prove that $\mathcal{A}$ is a non-empty open and closed subset of $[0,1]$.

\begin{enumerate}
\item  The set  $\mathcal{A}$ is not empty. This is because  $0\in\mathcal{A}$ since $v^0$ is the solution of $\mathcal{P}_0$. 

\item The set $\mathcal{A}$ is open in $[0,1]$. Given $t_0\in\mathcal{A}$, we need to prove that there is an $\epsilon>0$ such that $(t_0-\epsilon,t_0+\epsilon)\cap [0,1]\subset\mathcal{A}$. Define the map $T(t,u)=Q_t[u]$ for $t\in\r$ and $u\in  C^{2,\gamma}(\overline{\Omega})$. Then $t_0\in\mathcal{A}$ if and only if $T(t_0,u_{t_0})=0$. If we show that the derivative  of $Q_t$ with respect to $u$, say $(DQ_t)_u$, at the point $u_{t_0}$ is an isomorphism, it follows from the Implicit Function Theorem the existence of an open set $\mathcal{V}\subset C^{2,\gamma}(\overline{\Omega})$, with $u_{t_0}\in \mathcal{V}$, and a $C^1$ function $\xi:(t_0-\epsilon,t_0+\epsilon)\rightarrow \mathcal{V}$ for some $\epsilon>0$, such that $\xi(t_0)=u_{t_0}>0$ and  $T(t,\xi(t))=0$ for all $t\in (t_0-\epsilon,t_0+\epsilon)$: this guarantees that $\mathcal{A}$ is an open  set of  $[0,1]$.

The proof that $(DQ_t)_u$ is one-to-one is equivalent to prove that for any $f\in C^\gamma(\overline{\Omega})$, there is a unique solution $v\in C^{2,\gamma}(\overline{\Omega})$ of the linear equation $Lv:=(DQ_t)_u(v)=f$ in $\Omega$ and $v=\varphi$ on $\partial\Omega$. The computation of $L$ was done in Proposition \ref{pr-41}, obtaining 
$$Lv=(DQ_t)_uv=a_{ij}(Du)v_{ij}+\mathbf{b}_i(u,Du,D^2u)v_i+{\textbf c}(u,Du)v,$$
where $a_{ij}$  are defined in (\ref{op}) and 
$$\mathbf{b}_i=2\left( \Delta u-\frac{\alpha t}{u}\right)u_i+2u_ju_{ij},\quad {\textbf c}=\frac{\alpha t(1-|Du|^2)}{u^2}.$$
Since $\alpha<0$,  ${\textbf c}\leq 0$ and the existence and uniqueness is assured by standard theory (\cite[Th. 6.14]{gt}).

\item The set $\mathcal{A}$ is closed in $[0,1]$. Let $\{t_k\}\subset\mathcal{A}$ with $t_k\rightarrow t\in [0,1]$. For each $k\in\mathbb{N}$, there exists $u_{t_k}\in C^{2,\gamma}(\overline{\Omega})$, $u_{t_k}>0$,  such that $Q_{t_k}[u_{t_k}]=0$ in $\Omega$ and $u_{t_k}=\varphi$ in $\partial\Omega$. Define the set
$$\mathcal{S}=\{u\in C^{2,\gamma}(\overline{\Omega}): \exists t\in [0,1]\mbox{ such that }Q_{t}[u]=0 \mbox{ in }\Omega, u_{|\partial\Omega}=\varphi\}.$$
Then $\{u_{t_k}\}\subset\mathcal{S}$. If we prove that the set $\mathcal{S}$ is bounded in $C^{1,\beta}(\overline{\Omega})$ for some $\beta\in[0,\gamma]$, and since $a_{ij}=a_{ij}(Du)$ in (\ref{op}), then Schauder theory proves that $\mathcal{S}$ is bounded in $C^{2,\beta}(\overline{\Omega})$, in particular, $\mathcal{S}$ is precompact in $C^2(\overline{\Omega})$  (see Th. 6.6 and Lem. 6.36 in \cite{gt}). Thus there exists a subsequence $\{u_{k_l}\}\subset\{u_{t_k}\}$ converging in $C^2(\overline{\Omega})$ to some $u\in C^2(\overline{\Omega})$. Since $T:[0,1]\times C^2(\overline{\Omega})\rightarrow C^0(\overline{\Omega})$ is continuous, it follows $Q_t[u]=T(t,u)=\lim_{l\rightarrow\infty}T(t_{k_l},u_{k_l})=0$ in $\Omega$. Moreover, $u_{|\partial\Omega}=\lim_{l\rightarrow\infty} {u_{k_l}}_{|\partial\Omega}=\varphi$ on $\partial\Omega$, so $u\in C^{2,\gamma}(\overline{\Omega})$ and consequently, $t\in \mathcal{A}$.

The above reasoning asserts that   $\mathcal{A}$ is  closed in $[0,1]$ provided we find a constant $M$ independent of $t\in\mathcal{A}$,  such that  
$$\|u_t\|_{C^1(\overline{\Omega})}=\sup_\Omega |u_t|+\sup_\Omega|Du_t|\leq M.
$$
However the $C^0$ and $C^1$ estimates for the function $u_1$, that is, when the parameter  $t$ is $t=1$, are enough as we  now   see.  

The $C^0$ estimates for $u_t$ follow with the comparison principle. Indeed, let $t_1<t_2$, $t_i\in [0,1]$, $i=1,2$. Then $Q_{t_1}[u_{t_1}]=0$ and 
$$Q_{t_1}[u_{t_2}]=-\frac{(t_1-t_2)\alpha(1-|Du_{t_2}|^2)}{u_{t_2}}<0=Q_{t_1}[u_{t_1}]$$
because $\alpha<0$. Since $u_{t_1}=\varphi=u_{t_2}$ on $\partial\Omega$, the comparison principle yields $u_{t_1}<u_{t_2}$ in $\Omega$. This proves that the solutions $u_{t_i}$ are ordered in increasing sense according the parameter $t$. By (\ref{eh}), we find
\begin{equation}\label{ut}
\sup_\Omega u_t \leq \sup_\Omega u_1 \leq C_1.
\end{equation}

 In order to derive  the gradient estimates for the solution $u_t$,  the same computations  obtained  in   Proposition \ref{pr42}     conclude that $\sup_{\partial\Omega}|Du_t|$ is bounded by a constant depending on $\alpha$, $\Omega$, $\varphi$ and  $\|u_t\|_{0;\overline{\Omega}}$. Now  (\ref{ut}) implies that   the value  $\|u_t\|_{0;\overline{\Omega}}$ is bounded by $C_1$, which depends only on $\alpha$, $\varphi$ and $\Omega$, but not on $t$. 

\end{enumerate}
The above three steps prove the  existence part in Theorem \ref{t1}. The uniqueness is consequence  of Proposition \ref{pr-u} and this completes the proof of theorem.

 \end{proof}

A consequence of Theorem \ref{t1} is the solvability of   the Plateau problem if  $\alpha<0$ in the following situation.

\begin{corollary} Let $\Gamma$ be a spacelike $(n-1)$-submanifold of $\l^{n+1}$ with an one-to-one orthogonal projection $C$ on the hyperplane of equation $x_{n+1}=0$ such that $C$ is the boundary of a mean convex simply-connected domain $\Omega$.  Let $\alpha<0$. If $\Gamma$ has a spacelike extension to a graph on $\Omega$, then there exists a unique $\alpha$-singular maximal hypersurface $S$ spanning $\Gamma$. Moreover, $S$ is a graph on $\Omega$.
\end{corollary}

\begin{proof}   Theorem \ref{t1} asserts the existence of an $\alpha$-singular maximal hypersurface $S$ whose boundary is $\Gamma$ and $S$ is a graph on $\Omega$. Assume that $M$ is other  such a hypersurface. The property that  $M$ is spacelike implies that   the orthogonal projection $p:\r^{n+1}\rightarrow\r^n=\r^n\times\{0\}$, $p(x)=(x_1,\ldots,x_n)$ is a local diffeomorphism between $M$ and $\Omega$. In particular, $p:M\rightarrow\Omega$ is a covering map and since $\Omega$ is simply connected, the map $p$ is a diffeomorphism, in particular, $M$ is a graph on $\Omega$. Finally,  the uniqueness of \eqref{eqL} when $\alpha$ is negative concludes that $M=S$.
\end{proof}
 

\end{document}